\documentclass[reqno]{amsart}
\usepackage{amssymb}
\usepackage{amsmath}
\usepackage{color}

\makeatletter
\@addtoreset{equation}{section}
\makeatother

\usepackage[mathscr]{euscript}

\newtheorem{theorem}{Theorem}[section]
\newtheorem{lemma}[theorem]{Lemma}
\newtheorem{proposition}[theorem]{Proposition}
\newtheorem{corollary}[theorem]{Corollary}

\newtheorem{remark}[theorem]{Remark}

\newcommand{\mc}[1]{{\mathcal #1}}

\newcommand{\mb}[1]{{\mathbf #1}}
\newcommand{\bb}[1]{{\mathbb #1}}
\newcommand{\bs}[1]{{\boldsymbol #1}}

\newcommand{\<}{\langle}
\renewcommand{\>}{\rangle}

\begin{document}

\title[Large deviations of the exclusion process in two dimensions]
{A large deviations principle for the polar
  empirical measure in the two-dimensional symmetric simple exclusion
  process} \author{Claudio  Landim, Chih-Chung Chang, Tzong-Yow Lee}

\address{\noindent IMPA, Estrada Dona Castorina 110,
CEP 22460 Rio de Janeiro, Brasil and CNRS UPRES-A 6085,
Universit\'e de Rouen, 76128 Mont Saint Aignan, France.
\newline
e-mail:  \rm \texttt{landim@impa.br}
}

\address{Department of Mathematics, National Taiwan University,
Taipei, Taiwan, R.O.C.
\newline
e-mail: \rm \texttt{ccchang@math.ntu.edu.tw} }

\begin{abstract}
  We prove an energy estimate for the polar empirical measure of the
  two-dimensional symmetric simple exclusion process. We deduce from
  this estimate and from results in \cite{cll} large deviations
  principles for the polar empirical measure and for the occupation
  time of the origin.
\end{abstract}

\maketitle 

\section{Introduction}
\label{sec-n0}

We presented in \cite{cll} a large deviations principle for the
occupation time of the origin in the two-dimensional symmetric simple
exclusion process. The proof relies on a large deviations principle
for the ``polar'' empirical measure. After the paper was published and
after T-Y Lee passed away, A. Asselah pointed to us that there was a
flaw in the argument.  The proofs of the lower and upper bound of the
large deviations principle for the polar measure were correct, but the
bounds did not match.

We correct this inaccuracy in this article by showing that we may
restrict the upper bound to measures with finite energy, that is, to
absolutely continuous measures $\mu(dr) = m(r) dr$ whose density $m$
has a generalized derivative, denoted by $m'$, such that $\int_{\bb
  R_+} [m'(r)]^2/\sigma(m(r)) \, dr < \infty$, where
$\sigma(a)=a(1-a)$.

The large deviations principle of the occupation time of the origin is
correct as stated in \cite{cll}, and follows, through a contraction
principle, from the amended version of the large deviations principle
for the polar measure presented here.

There are many reasons to examine the large deviations of the
occupation time in dimension $2$. On the one hand, the unusual large
deviations decay rate $t/\log t$, with a logarithmic correction which
appears in critical dimensions. On the other hand, the unexpected
possibility to derive an explicit formula (cf. equation \eqref{b1}
below) for the large deviations rate function. Finally, the method by
itself may be of interest in other contexts. It has been shown
\cite{l} that in dimension $1$ the occupation time large deviations,
whose decay rate is $\sqrt{t}$, is related to the large deviations of
the empirical measure. Here, in dimension $2$, it is shown to be
connected to the large deviations of the polar measure. It is
conceivable that in higher dimensions, where the decay rate is $t$,
the large deviations are associated to some other type of empirical
measure.

We refer to \cite{cll} for further references and for an historical
background of this problem.  We wish to thank A. Asselah for pointing
to us the flaw in \cite{cll}, K. Mallick and K. Tsunoda for
stimulating discussion on occupation time large deviations and drawing
our attention to the recent papers \cite{s1, s2}. These exchanges
encouraged us to try to fill the gap left in \cite{cll}.

\section{Notation and results}
\label{sec-n1}

The speeded-up, symmetric simple exclusion process on $\bb Z^2$ is the
continuous-time Markov process on $\{0,1\}^{\bb Z^2}$ whose generator,
denoted by $L_T$, acts on functions $f: \{0,1\}^{\bb Z^2} \to \bb R$
which depends only on a finite number of coordinates as
\begin{equation*}
(L_T f) (\eta) \;=\; \frac T2 \sum_{j=1}^2  \sum_{x\in\bb Z^2} 
\big\{ f(\sigma^{x,x+e_j}\eta)-f(\eta) \big\}\; .  
\end{equation*}
In this formula, $\{e_1, e_2\}$ is the canonical basis of $\bb R^2$,
and $\sigma^{x,y}\eta$ is the configuration obtained from $\eta$ by
exchanging the occupation variables $\eta (x)$ and $\eta(y)$:
\begin{equation*}
(\sigma^{x,y}\eta) (z) \; =\;
\left\{
\begin{array}{ll}
\eta (z) & \text{ if $z\neq x$, $y$}, \\
\eta(x) & \text{ if $z= y$}, \\
\eta(y) & \text{ if $z= x$} \; .
\end{array}
\right.  
\end{equation*}

Denote by $\nu_\alpha$, $0\le\alpha\le 1$, the Bernoulli
product measure on $\{0,1\}^{\bb Z^2}$ with marginals given by
\begin{equation*}
\nu_\alpha\{\eta , \, \eta(x) =1 \} \; =\; \alpha\;, \quad
\text{for $x\in \bb Z^d$.} 
\end{equation*}
A simple computation shows that $\{\nu_\alpha,\, 0\le \alpha\le 1\}$
is a one-parameter family of reversible invariant measures. 

Denote by $D(\bb R_+, \{0,1\}^{\bb Z^2})$ the space of right
continuous functions $x: \bb R_+ \to \{0,1\}^{\bb Z^2}$ with left
limits, endowed with the Skorohod topology.  The elements of $D(\bb
R_+, \{0,1\}^{\bb Z^2})$ are represented by $\eta_s$, $s\ge 0$. Let
$\bb P_\alpha=\bb P_{T, \alpha}$, $0\le \alpha\le 1$, be the
probability measure on $D(\bb R_+, \{0,1\}^{\bb Z^2})$ induced by
Markov process whose generator is $L_T$ starting from $\nu_\alpha$.
Expectation with respect to $\bb P_\alpha$ is denoted by $\bb
E_\alpha$.

Denote by $\mc M$ the space of locally finite, nonnegative measures on
$(0,\infty)$.  Let $\sigma_T: \bb Z^2 \setminus\{0\} \to \bb R_+$ be
given by
\begin{equation*}
\sigma_T(x) \;=\; \frac{\log |x|}{\log T}\; ,
\end{equation*}
where $|x|$ represents the Euclidean norm of $x$, $|x|^2 = x^2_1 +
x^2_2$.  Denote by $\mu^{1,T}: \{0,1\}^{\bb Z^2} \to \mc M$ the
``polar'' empirical measure on $\bb R_+$ induced by a configuration
$\eta$:
\begin{equation*}
\mu^{1,T} (\eta) \;=\; \frac {1} {2\pi \log T} \sum_{x\in \bb
Z^2_*} \eta (x) \, \frac 1{|x|^2} \, \delta_{\sigma_T(x)} \; .
\end{equation*}
Here, $\delta_v$ is the Dirac measure concentrated on $v\in \bb R_+$.
Notice the factor $2\pi$ on the denominator to normalize the sum.
Denote by $\bar\mu^T : D(\bb R_+, \{0,1\}^{\bb Z^2}) \to \mc M$ the
measure on $\bb R_+$ obtained as the time integral of the measures
$\mu^{1,T}$:
\begin{equation}
\label{n-62}
\bar\mu^T \;=\; \int_0^1 \mu^{1,T} (\eta_s)\, ds\;.
\end{equation}
The main result of this article establishes a large deviations
principle for the measure $\bar\mu^T$ under $\bb P_\alpha$.

Denote by $1\!\!1$ the configuration in which all sites are occupied,
$1\!\!1 (x) = 1$ for all $x\in \bb Z^2$. The measures $\mu^{1,T}$,
$\bar\mu^T$ are nonnegative and bounded above by the measure
$\lambda_T = \mu^{1,T} (1\!\!1)$: for all nonnegative, continuous
function $H:(0,\infty) \to \bb R$ with compact support, and all
elements of $ \{0,1\}^{\bb Z^2}$, $D(\bb R_+, \{0,1\}^{\bb Z^2})$,
\begin{equation*}
0\;\le\; \int_{\bb R_+} H\, d\mu^{1,T} \;\le\;
\int_{\bb R_+} H\, d\lambda_T \quad\text{and}\quad
0\;\le\; \int_{\bb R_+} H\, d\bar \mu^{T} \;\le\;
\int_{\bb R_+} H\, d\lambda_T\;.
\end{equation*}
On the other hand, an elementary computation shows that there exists a
finite universal constant $C_0$ such that
\begin{equation}
\label{n-41}
\lambda_T ([a,b]) \;\le\; (b-a) \; +\; \frac{C_0}{\log T}  
\end{equation}
for all $0< a < b <\infty$, $T>1$. It is therefore natural to
introduce the space $\mc M_c$, $c>0$, of nonnegative, locally finite
measures $\mu$ defined on the Borel sets of $(0,\infty)$ and such that
$\mu ([a,b]) \le (b-a) + c$ for every $0< a < b <\infty$:
\begin{equation*}
\mc M_c \; =\; \Big\{ \mu\in\mc M : \mu ([a,b]) \le (b-a)
+ c \text{ for $0< a < b <\infty$ } \Big\} \;.
\end{equation*}
The uniform bound on the measure of the intervals makes the set $\mc
M_c$ endowed with the vague topology a compact, separable metric
space. Let $\mc M_0$ be the subspace of $\mc M_c$ of all measures which
are absolutely continuous with respect to the Lebesgue measure and
whose density is bounded by $1$. The subspace $\mc M_0$ is closed
(and thus compact).

Let $C_K((0,\infty))$ be the space of continuous functions $G:
(0,\infty)\to \bb R$ with a compact support, and let
$C^n_K((0,\infty))$, $n\ge 1$, be the space of compactly supported
functions $F: (0,\infty)\to \bb R$ whose $n$-th derivative is
continuous.  Denote by $\mc Q: \mc M_0 \to \bb R_+$ the energy
functional given by
\begin{equation}
\label{n-39}
\mc Q (m(r) \, dr) \;=\; \sup_{G\in C^1_K((0,\infty))}
\Big\{ -\, \int_0^\infty  G'(r)\, m(r) \, dr \;-\; 
\int_{0}^\infty \sigma (m(r)) \, G(r)^2 \, dr \Big\} \;,
\end{equation}
where $\sigma(a) = a(1-a)$ stands for the mobility of the exclusion
process.  By \cite[Lemma 4.1]{blm}, the functional $\mc Q$ is convex
and lower-semicontinuous. Moreover, if $\mc Q (m(r) \, dr)$ is finite,
$m$ has a generalized derivative, denoted by $m'$, and
\begin{equation*}
\mc Q (m(r) \, dr) \;=\; \frac{1}{4}\, \int_0^{\infty} 
\frac{[m'(r)]^2}{\sigma(m(r))} \, dr\;.
\end{equation*}

Fix $0<\alpha<1$, and let $\mc M_{0,\alpha}$ the space of measures in
$\mc M_0$ whose densities are equal to $\alpha$ on $(1/2,\infty)$:
$\mc M_{0,\alpha} = \{\mu(dr) = m(r) dr \in\mc M_0 : m(r)=1/2 \text{
  a.s.\! in } (1/2,\infty) \}$. Denote by $I_{\mc Q, \alpha}: \mc M_c \to \bb
R_+$ the functional given by
\begin{equation}
\label{n-52}
I_{\mc Q, \alpha} (\mu) \;=\; 
\begin{cases}
\pi\, \mc Q (\mu) & \text{ if $\mu\in \mc M_{0,\alpha}$,} \\
+\infty & \text{otherwise.} 
\end{cases}
\end{equation}
Since the set $\mc M_{0,\alpha}$ is convex and closed, the functional
$I_{\mc Q, \alpha}$ inherits from $\mc Q$ the convexity and the
lower-semicontinuity. Furthermore, as $\mc M_c$ is compact and $I_{\mc
  Q, \alpha}$ lower semi-continuous, the level sets of $I_{\mc Q,
  \alpha}$ are compact. Next assertion is the main result of the
article.

\begin{theorem}
\label{mt01}
For every closed subset $F$ of $\mc M_c$ and every open subset
$G$ of $\mc M_c$,
\begin{align*}
& \limsup_{T\to\infty} \frac{\log T} T \log \bb P_{\alpha}
\big[ \bar\mu^T \in F \big] \;\le\; - \inf_{\mu \in F} I_{\mc Q, \alpha}
(\mu)\; , \\
& \quad \liminf_{T\to\infty} \frac{\log T} T \log \bb P_{\alpha}
\big[ \bar\mu^T \in G \big] \;\ge\; - \inf_{\mu \in G} I_{\mc Q, \alpha}
(\mu)\; .
\end{align*}
Moreover, the rate functional $I_{\mc Q, \alpha} : \mc M_c \to \bb
R_+$ is convex, lower semi-continuous and has compact level sets.
\end{theorem}

\begin{remark}
\label{last}
We explain in this remark the flaw in \cite{cll}.  Denote by $\hat
I_{\alpha}: \mc M_0 \to \bb R_+$ the functional given by
\begin{equation*}
\hat I_{\alpha} (m(r) \, dr) \;=\; \sup_{G\in C^1_K((0,1/2))}
\Big\{ -\, \int_0^{1/2}  G'(r)\, m(r) \, dr \;-\; 
\int_{0}^{1/2} \sigma (m(r)) \, G(r)^2 \, dr \Big\} \;.
\end{equation*} 
Note that the supremum is carried over functions whose support is now
contained in $(0,1/2)$. Let $I_{\alpha}: \mc M_c \to \bb R_+$ be given
by
\begin{equation}
\label{n-61}
I_{\alpha} (\mu) \;=\; 
\begin{cases}
\hat I_{\alpha} (\mu) & \text{ if $\mu\in \mc M_{0,\alpha}$,} \\
+\infty & \text{otherwise.} 
\end{cases}  
\end{equation}
Section 5 of \cite{cll} shows that $I_\alpha$ is an upper bound for
the large deviations principle. This upper bound is not
sharp. Consider, for instance, the measure $\mu_\beta(dr) = m_\beta(r)
dr$, $\beta\not = \alpha$, where $m_\beta(r) = \beta$ for $0\le r <
1/2$ and $m_\beta(r) = \alpha$ for $r\ge 1/2$. By \eqref{n-61},
$I_\alpha(\mu_\beta)=0$, which is clearly not sharp.

The problem lies in the proof of Lemma 6.3, at the end of page 686. It
is claimed there that if $I_\alpha(\mu)<\infty$ for an absolutely
continuous measure $\mu(dr) = m(r) dr$, there exists a sequence of
smooth functions $m_n$ such that $m_n(r) = \alpha$ for $r\ge 1/2$,
$\mu_n(dr) = m_n(r) dr \to \mu$ in the vague topology, and
$I_\alpha(m_n(r) dr)\to I_\alpha(\mu)$. This is not true for the
measure $\mu_\beta$ introduced in the previous paragraph.
\end{remark}

\begin{remark}
\label{n-l19}
For a measure $\mu$ in $\mc M_{0,\alpha}$ with finite energy,
$\mc Q(\mu)<\infty$,
\begin{equation*}
I_{\mc Q, \alpha} (\mu) \;=\; \frac{\pi}{4}\, \int_0^{1/2} 
\frac{[m'(r)]^2}{\sigma(m(r))} \, dr \;=\; I_{\alpha} (\mu)\;.
\end{equation*}
However, for measures in $\mc M_{0,\alpha}$ with inifinite energy,
$I_{\mc Q, \alpha}(\mu)=\infty$, while $I_{\alpha}(\mu)$ might be
finite. For example, $I_{\mc Q, \alpha} (\mu_\beta) = \infty$ and
$I_{\alpha}(\mu_\beta) = 0$, where $\mu_\beta$ is the measure
introduced in the previous remark.
\end{remark}

This remark shows that what is missing in the proof of the large
deviations principle in \cite{cll} is the derivation of the property
that measures with infinite energy in $\bb R_+$ have infinite
cost. Note that for measures $\mu(dr) = m(r) \, dr$ in $\mc
M_{0,\alpha}$ and such that $I_{\alpha}(\mu) <\infty$, the finiteness
of the energy is a property of the measure in a vicinity of $1/2$
because $m(r)=\alpha$ for $r>1/2$ and the energy of $\mu$ on the
interval $(0,1/2)$ is finite by definition of $I_{\alpha}$.

Corollary \ref{n-l15} below asserts that measures with infinite energy
in $\bb R_+$ have infinite cost. Its proof relies on Proposition
\ref{n-l6}, a new result which states that a superexponential
two-blocks estimate for the cylinder function $[\eta(0) -
\eta(e_j)]^2$ holds on the entire space $\bb Z^2$, and not only on
$\{x\in \bb Z^2 : |x| < T^{1/2-\delta}\}$.  Proposition \ref{n-l6} is
restricted to the local function $[\eta(0) - \eta(e_j)]^2$ because the
concavity of the map $\beta \mapsto E_{\nu_\beta}[\{\eta(0) -
\eta(e_j)\}^2]$ is used. We refer to Remark \ref{n-l18} for further
comments on this result.

Remark \ref{n-l11} explains why it is possible to prove an energy
estimate in the whole space $\bb Z^2$, but it is not possible to
handle, in the proof of the large deviations upper bound,
perturbations defined in the entire space. Actually, in the upper
bound, the dynamics is perturbed only in a ball $\{x\in \bb Z^2 : |x|
< T^{1/2-\delta}\}$.

The energy estimate for the empirical measure $\bar\mu^T$ requires
some care because the measure is defined on $\bb R_+$, a
one-dimensional space, and the system evolves the two-dimensional
space $\bb Z^2$. This difficulty is surmounted through formula
\eqref{n-08} and Lemma \ref{n-l9}.

A large deviations principle for the occupation time of the origin
follows from Theorem \ref{mt01} and a contraction principle. The proof
of this result is presented in Section 7 of \cite{cll}. We recall it
here since it is one of the main motivations for Theorem \ref{mt01}.

\begin{theorem}
\label{s0}
For every closed subset $F$ of $[0,1]$ and every open subset $G$ of
$[0,1]$,
\begin{align*}
& \limsup_{T\to\infty} \frac{\log T} T \log \bb P_{\alpha}
\big[ \int_0^1 \eta_s(0)\, ds  \in F \big] 
\;\le\; - \, \inf_{\beta \in F} \Upsilon_\alpha (\beta)\; , \\
& \quad \liminf_{T\to\infty} \frac{\log T} T \log \bb P_{\alpha}
\big[ \int_0^1 \eta_s(0)\, ds \in G \big] \;\ge\; 
- \, \inf_{\beta \in G} \Upsilon_\alpha (\beta)\; ,
\end{align*}
where $\Upsilon_\alpha: [0,1] \to\bb R_+$ is the rate function given
by
\begin{equation*}
\Upsilon_\alpha (\beta) \; =\; \frac \pi 2 \Big\{\sin^{-1}(2\beta
-1)-\sin^{-1}(2\alpha -1)\Big\}^2 \; .
\end{equation*}
\end{theorem}

Actually the rate function $\Upsilon_\alpha$ is derived through the
variational problem
\begin{equation}
\label{b1}
\Upsilon_\alpha (\beta) \; =\; \inf_{m} \frac{\pi}{4}
\int_0^{1/2} \frac {m'(r)^2}{\sigma(m(r))} \, dr \; ,
\end{equation}
where the infimum is carried over all smooth functions $m:[0,1/2]\to
\bb R$ such that $m(0) = \beta$, $m(1/2) = \alpha$.

The article is organized as follows. In Section \ref{sec n-3}, we state
the superexponential estimates and in the following one the energy
estimate. In Section \ref{nsec-1}, we present an alternative formula
for the large deviations rate functional and derive some of its
properties. In Sections \ref{sec3} and \ref{n-sec6} we prove the upper
bound and the lower bounds of the large deviations principle.

\section{Superexponential estimate} 
\label{sec n-3}

We present in this section some superexponential estimates needed in
the proof of the large deviations principle. We start with an
elementary estimate. Denote by $\phi$ the approximation of the
identity given by $\phi(r) = (1/2) \mb 1\{[-1,1]\}$, where $\mb 1\{
[-1,1] \}$ represents the indicator of the interval $[-1,1]$. For $r$,
$\delta >0$, let $\phi_{r,\delta}$ be the family of approximations
induced by $\phi$: $\phi_\delta (s) = \delta^{-1} \phi (s/\delta)$,
$\phi_{r,\delta} (s) = \phi_{\delta} (s-r)$.

It will be simpler to work with a continuous family of approximations
of the identity. Let $\psi_\delta$ be a nonnegative, continuous
function, bounded by $1/(2\delta)$, which coincides with $\phi_\delta$
on $[-(\delta - \delta^2),(\delta - \delta^2)]$ and whose support is
contained in $[-(\delta + \delta^2),(\delta + \delta^2)]$. Set
$\psi_{r,\delta} (s) = \psi_{\delta} (s-r)$.

Denote by $\mu(f)$ the integral of a continuous and compactly
supported function $f:(0,\infty) \to \bb R$ with respect a the measure
$\mu\in \mc M_c$:
\begin{equation*}
\mu(f) \;=\; \int_{\bb R_+} f(r)\, \mu(dr)\;.
\end{equation*}
By construction, for all $b>0$, there exists a finite constant $C =
C(b)$ such that
\begin{equation}
\label{n-29}
\limsup_{T\to\infty} \sup_{2\delta \le r \le b} \sup_{\eta}
\big| \, \mu^{1,T}(\phi_{r,\delta}) - 
\mu^{1,T}(\psi_{r,\delta}) \, \big| \;\le\; C \delta\;.
\end{equation}
A similar estimate holds if $\mu^{1,T}$ is replaced by $\bar\mu^{T}$.
Note that for each measure $\mu\in \mc M_c$, $\mu(\psi_{r,\delta})$ is a
continuous function of the parameter $r$ because $\psi_{\delta}$ is a
bounded, continuous function.

The next comparison between a Riemann sum with its integral
counterpart will also be used repeatedly.  

\begin{lemma}
\label{n-l22}
Let $H: \bb R_+ \to \bb R$ be a Lipschitz-continuous function with
compact support in $(a,b)$, $0<a<b<\infty$. Then, there exists a
finite constant $C_0$ depending only on $\Vert H \Vert_\infty =
\sup_{r\in\bb R_+} |H(r)|$ and on the Lipschitz constant of $H$ such that
\begin{equation*}
\Big|\, \frac 1{\log T} \sum_{x\in \bb Z^2_*} H(\sigma_T(x)) \,
\frac 1{|x|^2} \; -\; 2\pi \int_{\bb R_+} H(r) \, dr \,\Big|\;\le\;
\frac{C_0}{T^a}\;\cdot
\end{equation*}
\end{lemma}

\begin{proof}
Let $\square_x = [x_1,x_1+1) \times [x_2,x_2+1)$. The proof consists
in comparing 
\begin{equation*}
H(\sigma_T(x)) \, \frac 1{|x|^2} \quad{with}\quad
\int_{\square_x} H(\sigma_T(z)) \, \frac 1{|z|^2}\, dz\;,
\end{equation*}
and then the sum over $x$ of the second term in this formula with the
$2\pi \int_{\bb R_+} H(r) \, dr$. Details are left to the reader.
\end{proof}

\begin{remark}
\label{n-l23}
Let $J$ be a function in $C^1_K((0,\infty))$.  We will apply the
previous result to $H(r) = J(r) \, \mu^{1,T} (\psi_{r,\delta})$ in
Lemma \ref{n-l14} below and to $H(r) = J(r)^2\, \sigma (\mu^{1,T}
(\psi_{r,\delta}))$ in Corollary \ref{n-l8}. The proof of Lemma
\ref{n-l22} relies on the finiteness of $\Vert H\Vert_\infty$ and on
the Lipschitz property of $H$. Both conditions are fulfilled by the
map $r\to \mu^{1,T} (\psi_{r,\delta})$ on compact intervals of
$(0,\infty)$. On the one hand, for all $0<2\delta<r$, $0\le \mu^{1,T}
(\psi_{r,\delta}) \le 1$. On the other hand, as $\psi_\delta$ is a
Lipschitz-continuous function, by definition of $\psi_{r,\delta}$, for
each $\delta>0$, there exists a finite constant $C(\delta)$ such that
$|\mu^{1,T} (\psi_{r,\delta}) - \mu^{1,T} (\psi_{s,\delta})|\le
C(\delta) |r-s|$ for all $r$, $s\ge 2\delta$. Thus, Lemma \ref{n-l22}
holds for these functions with a contant $C_0$ which also depends on
$\delta$.
\end{remark}

The next estimate will be used to introduce space averages through the
regularity of the test function and a summation by parts.  Denote by
$\bb A_{T,\delta}(x) \subset \bb Z^2$, $\delta>0$, $|x|>T^{2 \delta}$,
the annulus
\begin{equation*}
\bb A_{T,\delta}(x)  \;=\; \big\{ y\in \bb Z^2 : |x|\,
T^{-\delta} \le |y| \le |x| \, T^{\delta} \big\}\;.
\end{equation*}
Let $J \in C_K((0,\infty))$ be a Lipschitz continuous function whose
support is contained in $[a,b]$. There exists a finite constant
$C(J)$, depending only on $J$, such that for all $0<\delta\le a/2$,
\begin{equation}
\label{n-23}
\limsup_{T\to\infty} \sup_{x\in \bb Z^2}
\Big|\, J (\sigma_T(x)) -
\frac 1 {4\pi \delta \log T} \sum_{y\in \bb A_{T,\delta}(x) } 
\frac {1}{|y|^2}  \, J (\sigma_T(y))\, \Big| \;\le\; C(J) 
\, \delta \;.
\end{equation}
Note that we may restrict the supremum to the points $x$ such that
$|x|\ge T^{a-\delta}$. 

\begin{lemma}
\label{n-l14}
Let $J \in C^1_K((0,\infty))$. There exists a finite constant $C_0$,
depending only on $J$, such that 
\begin{equation*}
\limsup_{T\to\infty} \sup_{\eta} \Big|\, 
\int_{\bb R_+} J(r)\, \mu^{1,T} (dr) \,-\, \int_{\bb R_+} J(r) \, 
\mu^{1,T} (\psi_{r,\delta})  \, dr \Big| \;\le \; 
C_0\, \delta \;,
\end{equation*}
A similar result is in force with $\mu^{1,T}$ replaced by $\bar
\mu^{T}$.
\end{lemma}

\begin{proof}
This result is a simple consequence of \eqref{n-23}, a summation
by parts, the bound \eqref{n-29} and Remark \ref{n-l23}.
\end{proof}

We continue with two lemmata whose proofs are similar to the one of
Lemma 5.1 in \cite{cll}.

\begin{lemma}
\label{n-l1}
For every $\varrho>0$ and continuous function $H:[1/2,\infty)\to\bb R$
with compact support,
\begin{equation*}
\limsup_{T\to\infty} \frac{\log T}T \log \bb P_{\alpha} \Big[ \, \Big|
\int_{\bb R_+} H(r)\,  \bar\mu^T(dr) 
\,-\, \alpha \int_{1/2}^\infty H(r)\, dr \, \Big| 
\, > \varrho \, \Big] \; =\; - \infty\;.
\end{equation*}
\end{lemma}

Let $A$ be a finite subset of $\bb Z^2$, and denote by $\eta_A$ the
local function $\eta_A = \prod_{x\in A} \eta(x)$. For a continuous
function $H: (0,\infty) \to\bb R$ with compact support and $j=1$, $2$,
let
\begin{equation*}
W^{H,A}_{T,j} (\eta) \;=\; \frac 1{\log T} \sum_{x\in \bb Z^2_*} 
\frac 1{|x|^2} \, \frac {x_j^2}{|x|^2} \, H(\sigma_T(x))\, 
\{ \eta_{A+x} - \alpha^{|A|} \} \;,
\end{equation*}
where $A+x = \{y+x: y\in A\}$.


\begin{lemma}
\label{n-l2}
For every finite subset $A$ of $\bb Z^2$, $\varrho>0$, $j=1$, $2$, and
continuous function $H: [1/2,\infty) \to\bb R$ with compact support,
\begin{equation*}
\limsup_{T\to\infty} \frac {\log T}T \log \bb P_{\alpha} \Big[ \, \Big| 
\int_0^1 W^{H,A}_{T,j} (\eta_s)  \, ds \, \Big| \, > \varrho \, \Big]
\; =\; - \infty\;. 
\end{equation*}
\end{lemma}

Since any local function can be expressed as a linear combination of
function of type $\eta_A$ this result extends to all local functions.

Consider a continuous, non-negative function $J: \bb R_+ \to \bb R$
with compact support in $(0,\infty)$.  Let
$W_{T}^{J,\delta}$ be the local function defined as
\begin{equation}
\label{n-19}
W_{T}^{J,\delta} (\eta) \;=\;
\frac 1{\log T} \sum_{j=1}^2
\sum_{x\in \bb Z} J (\sigma_T(x)) \, \frac
{x^2_j}{|x|^4}  \, \Big\{ \big[\eta(x+e_j) - \eta(x)\big]^2
\;-\; 2 \, \sigma \big(m_{\delta,T} (x,\eta) \big)\Big\}\;,
\end{equation}
where
\begin{equation*}
m_{\delta,T} (x,\eta) \;=\;
\begin{cases}
\mu^{1,T}\big( \psi_{\sigma_T(x), \delta} \big) 
& \text{ if $\sigma_T(x) < 1/2$,} \\
\alpha & \text{ if $\sigma_T(x) \ge 1/2$.}
\end{cases}
\end{equation*}

\begin{proposition}
\label{n-l6}
Let $J: \bb R_+ \to \bb R$ be a non-negative function of class $C^1$
with compact support in $(0,\infty)$. For every $\varrho>0$, 
\begin{equation*}
\limsup_{\delta \to 0} \limsup_{T\to\infty} \frac {\log T}T 
\, \log \bb P_{\alpha} \Big[ \, 
\int_0^1 W_{T}^{J,\delta} (\eta_s) \, ds \, \, > \varrho \,
\Big] \; =\; - \infty\;. 
\end{equation*}
\end{proposition}

\begin{remark}
\label{n-l18}
The concavity of the mobility $\sigma$ plays an important role in the
proof of this proposition. We are not able to prove the so-called
superexponential two-blocks estimate, but only a mesoscopic
superexponential estimate. The concavity of $\sigma$ permits to insert
inside $\sigma$ macroscopic averages through Jensen's inequality. This
argument provides an inequality which, fortunately, goes in the right
direction.

For the same reasons, we are not able to prove this proposition for
the absolute value of the time integral.
\end{remark}

The proof of this proposition is divided in several steps.  Denote
by $\bb R^{(1)}_{T,\kappa}$, $\bb R^{(2)}_{T,\kappa}$, $\bb
R^{(3)}_{T}$, $0<\kappa<1/2$, the subsets of $\bb Z^2$ defined by
\begin{gather*}
\bb R^{(1)}_{T,\kappa} \;=\;
\{x\in \bb Z^2: \sigma_T(x) < 1/2 -\kappa\}\;, \quad
\bb R^{(3)}_{T} \;=\;
\{x\in \bb Z^2: \sigma_T(x) > 1/2\} \;, \\
\bb R^{(2)}_{T,\kappa} \;=\;
\{x\in \bb Z^2: (1/2)-\kappa \le \sigma_T(x) \le 1/2\}\;. \\
\end{gather*}

\smallskip\noindent{\bf A. The region $\bb R^{(3)}_{T}$.} On the
region $\bb R^{(3)}_{T}$, $m_{\delta,T} (x,\eta) = \alpha$. 
Let $W_{T}^{(3), J}$ be the local function defined as
\begin{equation*}
W_{T}^{(3), J} (\eta) \;=\; \frac 1{\log T} \sum_{j=1}^2
\sum_{x\in \bb R^{(3)}_{T}} J (\sigma_T(x)) \, \frac
{x^2_j}{|x|^4}  \, \Big\{ \big[\eta(x+e_j) - \eta(x)\big]^2
\;-\; 2 \, \sigma (\alpha)\Big\}\;.
\end{equation*}
By Lemma \ref{n-l2}, for every $\varrho>0$,
\begin{equation}
\label{n-17}
\limsup_{T\to\infty} \frac {\log T}T 
\, \log \bb P_{\alpha} \Big[ \, \Big|
\int_0^1 W_{T}^{(3), J} (\eta_s) \, ds 
\, \Big| \, > \varrho \, \Big] \; =\; - \infty\;.
\end{equation}

\smallskip\noindent{\bf B. The region $\bb R^{(2)}_{T}$.}  Let $\kappa
= \kappa(J,c)$, $c>0$, be such that
\begin{equation}
\label{n-18}
\frac 1{\log T}\, \sum_{x\in \bb R^{(2)}_{T,\kappa}} 
\frac{1}{|x|^2}\, J( \sigma_T(x) ) \;\le\; c\;\cdot
\end{equation}
Taking $\kappa = \kappa(J,\varrho/3)$ in the definition of the regions
$\bb R^{(i)}_{T}$, the contribution of the region $\bb R^{(2)}_{T}$ to
the sum defining $W_{T,j}^{J,\delta} (\eta_s)$ is bounded in
absolute value by $\varrho/3$ because the absolute value of the
expression inside braces in \eqref{n-19} is bounded by $1$.

\smallskip\noindent{\bf C. The region $\bb R^{(1)}_{T}$.}  Denote by
$\bb T_\pi$ the one-dimensional torus $[-\pi,\pi)$, and by $\Theta(x)$
the angle of $x\in \bb R^2\setminus \{0\}$ so that $x = (\, |x|\, \cos
\Theta(x) \,,\, |x|\, \sin \Theta(x)\, )$.

Fix a positive function $q : (0,1] \to (0,1]$ decreasing to $0$ slower
than the identity: $\lim_{\varepsilon\to 0}
q(\varepsilon)/\varepsilon=\infty$.  For $0<\varepsilon <r_0$, let
\begin{align*}
& \iota_+ \;=\; \iota_+(\varepsilon , r_0, T)
\;=\; \frac 1{\log T} \log \frac{T^{r_0} + T^{\varepsilon}}{T^{r_0}} \; , \\
& \quad \iota_- \;=\; \iota_-(\varepsilon , r_0, T)
\;=\; \frac 1{\log T} \log \frac{T^{r_0}}{T^{r_0} - T^{\varepsilon}}\;\cdot
\end{align*}
Denote by $M^{r,\theta}_{T, \varepsilon} (\eta)$, $\theta \in \bb
T_\pi$, the weighted average of particles in the polar cube $[r
- \iota_- , r + \iota_{+}] \times [\theta - q(\varepsilon) , \theta+
q(\varepsilon)]$ for a configuration $\eta$:
\begin{equation}
\label{n-16}
M^{r,\theta}_{T, \varepsilon} (\eta)
\;=\; \frac {1} {2 (\iota_{+} + \iota_-) q(\varepsilon) \log T}
\sum_{\substack{T^{r}-T^\varepsilon\le |z| \le
    T^r+T^\varepsilon\\ \theta - q(\varepsilon) \le \Theta(z) \le
    \theta + q(\varepsilon) }}
\frac {\eta (z)}{|z|^2}\; \cdot
\end{equation}
Note that this average is performed over a mesoscopic polar square.

Let $W_{T}^{J,\kappa,\varepsilon}$ be the local function defined as
\begin{equation*}
W_{T}^{J,\kappa,\varepsilon} (\eta) \;=\; \frac 1{\log T}
\sum_{x\in \bb R^{(1)}_{T,\kappa}} \sum_{j=1}^2 J (\sigma_T(x)) \, \frac
{x^2_j}{|x|^4}  \, \Psi_{T,j,\varepsilon}(x,\eta) \;,
\end{equation*}
where $\Psi_{T,j,\varepsilon}(x,\eta)$ is given by
\begin{equation*}
\Psi_{T,j,\varepsilon}(x,\eta) \;=\; \big[\eta(x+e_j) - \eta(x)\big]^2
\;-\; 2 \, \sigma \Big( M^{\sigma_T(x), \Theta(x)}_{T,\varepsilon}\Big)\;.
\end{equation*}
Next lemma is the superexponential estimate presented in Lemma 4.1 of
\cite{cll}. 

\begin{lemma}
\label{n-l7}
For any function $J$ satisfying the assumptions of Proposition
\ref{n-l6}, $0<\kappa<1/2$, and $\varrho>0$, 
\begin{equation*}
\limsup_{\varepsilon \to 0}\limsup_{T\to \infty}
\frac{\log T}T \log \bb P_{\alpha} \Big[ \, \Big\vert
\int_0^1  W_{T}^{J,\kappa,\varepsilon} (\eta_s)\, ds
\Big\vert \, >\, \varrho \, \Big] \; =\; - \infty\; .
\end{equation*}
\end{lemma}

To replace the average $M^{\sigma_T(x), \Theta(x)}_{T,\varepsilon}$
over a mesoscopic square by a macroscopic object we use the concavity
of $\sigma$. Fix a non-negative, Lipschitz-continuous function $J: \bb
R_+ \to \bb R$ whose support is contained in $(a,b)$ for
$0<a<b<\infty$. We claim that there exists a contant $C_0$, depending
only on $J$, such that for all $0< \varepsilon \le \delta < \kappa/2$,
$\kappa\le \min\{1/2, a\}$,
\begin{equation}
\label{n-15}
\begin{aligned}
& \frac 2 {\log T} \sum_{j=1}^2
\sum_{x\in \bb R^{(1)}_{T,\kappa}} J (\sigma_T(x)) \, \frac
{x^2_j}{|x|^4}  \, \sigma \Big( M^{\sigma_T(x),
  \Theta(x)}_{T,\varepsilon}\Big) \\
&\qquad \;\le\; \frac {2} {\log T} 
\sum_{x\in \bb R^{(1)}_{T,\kappa}} J (\sigma_T(x)) \, \frac
{1}{|x|^2}  \, \sigma \big(m_\delta (\sigma_T(x),\eta)\big)
\;+\; C_0 \delta \;+\; o_T(1) \;,
\end{aligned}
\end{equation}
where $o_T(1) \to 0$ as $T\to\infty$, uniformly over $\eta$.

To prove this assertion, on the left-hand side, sum in $j$ to replace
$\sum_j x^2_j$ by $|x|^2$. Add and subtract an average of $J$ over the
set $\bb A_{T,\delta}(x)$ introduced just above \eqref{n-23}. By
\eqref{n-23}, performing a summation by parts, we conclude that for
$\delta\le \kappa/2$, the left-hand side of \eqref{n-15} is bounded
above by
\begin{equation*}
\frac 2 {\log T} \sum_{y\in \bb R^{(1)}_{T,\kappa-\delta}} 
\frac {1}{|y|^2} \, J (\sigma_T(y)) \,
\frac 1 {4 \pi \delta \log T} \sum_{x\in \bb A_{T,\delta}(y) } 
\frac {1}{|x|^2} \, \sigma \Big( M^{\sigma_T(x), \Theta(x)}_{T,\varepsilon}\Big) 
\;+\; C(J) \, \delta\;,
\end{equation*}
where the sum over $x$ is also restricted to the set $\bb
R^{(1)}_{T,\kappa}$. 

The sum for $y$ such that $T^{(1/2) -\kappa-\delta} \le |y| \le
T^{(1/2) -\kappa+\delta}$ is bounded by $C(J) \delta$. We may thus
remove these terms from the sum by paying this price. For $y$ such
that $|y| \le T^{(1/2) -\kappa-\delta}$ we may remove in the second
sum the restriction that $x \in \bb R^{(1)}_{T,\kappa}$. After
removing this restriction we may insert in the first sum the term $y$
such that $T^{(1/2) -\kappa-\delta} \le |y| \le T^{(1/2) -\kappa}$ by
paying an extra error bounded by $C(J) \delta$. This shows that the
previous sum is less than or equal to
\begin{equation*}
\frac 2 {\log T} \sum_{y\in \bb R^{(1)}_{T,\kappa}} 
\frac {1}{|y|^2} \, J (\sigma_T(y)) \,
\frac 1 {4 \pi \delta \log T} \sum_{x\in \bb A_{T,\delta}(y) } 
\frac {1}{|x|^2} \, \sigma \Big( M^{\sigma_T(x), \Theta(x)}_{T,\varepsilon}\Big) 
\;+\; C(J) \, \delta\;.
\end{equation*}

Substituting $x$ by $y$, as $\sigma$ is concave, the previous
expression is bounded above by
\begin{equation*}
\frac 2 {\log T} \sum_{x\in \bb R^{(1)}_{T,\kappa}} 
\frac {1}{|x|^2} \, J (\sigma_T(x)) \, \sigma \Big(
\frac 1 {4\pi \delta \log T} \sum_{y\in \bb A_{T,\delta}(x) } 
\frac {1}{|y|^2} \, M^{\sigma_T(y), \Theta(y)}_{T,\varepsilon}\Big) 
\;+\; C(J) \, \delta\;.
\end{equation*}
Recall the definition \eqref{n-16} of $M^{\sigma_T(y),
  \Theta(y)}_{T,\delta}$ and to sum by parts inside $\sigma$ to bound
the previous expression by.
\begin{equation*}
\frac 2 {\log T} \sum_{x\in \bb R^{(1)}_{T,\kappa}} 
\frac {1}{|x|^2} \, J (\sigma_T(x)) \, \sigma \Big(
\mu^{1,T}\big( \phi_{\sigma_T(x), \delta} \big) \Big) 
\;+\; C(J) \, \delta\;.
\end{equation*}
To complete the proof of \eqref{n-15}, it remains to recall
\eqref{n-29} to replace $\phi_{\sigma_T(x), \delta}$ by
$\psi_{\sigma_T(x), \delta}$ inside $\sigma$. \smallskip

We summarize in Lemma \ref{n-l12} below the estimate obtained in the
region $\bb R^{(1)}_{T,\kappa}$. The statement requires some notation.
For a continuous function $J: \bb R_+ \to \bb R$ with compact support
in $(0,\infty)$, $\delta>0$, $\kappa>0$, let
\begin{equation}
\label{n-27}
\begin{gathered}
W_{T,\kappa}^{J,\delta} (\eta) \;=\; \frac 1{\log T} \sum_{j=1}^2
\sum_{x\in \bb R^{(1)}_{T,\kappa}} J (\sigma_T(x)) \, \frac
{x^2_j}{|x|^4}  \, \Psi_{\delta, T}(j, x,\eta) \;, \\
\text{where} \quad \Psi_{\delta, T}(j, x,\eta) \;=\;
\big[\eta(x+e_j) - \eta(x)\big]^2
\;-\; 2 \, \sigma \big(m_{\delta,T} (x,\eta) \big)\;,
\end{gathered}
\end{equation}
and $m_{\delta,T} (x,\eta)$ is defined below \eqref{n-19}.  Next lemma
follows from Lemma \ref{n-l7} and \eqref{n-15} by taking $\delta =
\varepsilon$.

\begin{lemma}
\label{n-l12}
Let $J: \bb R_+ \to \bb R$ be a non-negative function of class $C^1$
with compact support in $(0,\infty)$. For every $\varrho>0$,
$0<\kappa<1/2$, 
\begin{equation*}
\limsup_{\delta \to 0} \limsup_{T\to\infty} \frac {\log T}T 
\, \log \bb P_{\alpha} \Big[ \, 
\int_0^1 W_{T, \kappa}^{J,\delta} (\eta_s) \, ds \, \, > \varrho \,
\Big] \; =\; - \infty\;. 
\end{equation*} 
\end{lemma}

\begin{proof}[Proof of Proposition \ref{n-l6}]
Fix a function $J$ and $\varrho>0$ and recall the definition of the
regions $\bb R^{(1)}_{T,\kappa}$, $\bb R^{(2)}_{T,\kappa}$, $\bb
R^{(3)}_{T,\kappa}$, introduced just below the statement of the
proposition. Let $0<\kappa = \kappa(J,\varrho/3)$ for \eqref{n-18} to
hold with $c=\varrho/3$.  Fix this $\kappa$ and decompose the sum over
$x$ in \eqref{n-19} according to these $3$ regions.

By definition of $\kappa$, the sum over the region $\bb
R^{(2)}_{T,\kappa}$ is bounded by $\varrho/3$. Assertion \eqref{n-17}
takes care of the region $\bb R^{(3)}_{T,\kappa}$ and Lemma
\ref{n-l12} of the region $\bb R^{(1)}_{T,\kappa}$.
\end{proof}

\section{Energy estimate}
\label{sec n-2}

We prove in this section a microscopic energy estimate. It follows
from this result that measures with infinite energy have infinite cost
in the large deviations principle. This crucial point in the proof of
the large deviations dates back to \cite{qrv}.

The following elementary observation will repeatedly be used in the
sequel. For any sequence $M_T\to\infty$, and positive sequences $a_T$,
$b_T$, 
\begin{equation}
\label{n-21}
\limsup_{T\to\infty} \frac 1{M_T} \log (a_T + b_T) =
\max\Big \{\limsup_{T\to\infty} \frac 1{M_T} \log a_T \, , \,
\limsup_{T\to\infty} \frac 1{M_T}  \log b_T \Big \}  \;.
\end{equation}

The Dirichlet form of a function also plays a role in this section.
For a local function $f$, denote by $\mc E_T (f) = \mc E_{T,\alpha}
(f)$ the Dirichlet form of $f$:
\begin{equation*}
\mc E_T (f) \;=\; \< \, f \,,\, (-L_T) f \, \>_{\nu_\alpha}\;,
\end{equation*}
where $\< f \,,\, g \>_{\nu_\alpha}$ represents the scalar product in
$L^2(\nu_\alpha)$. An elementary computation provides an explicit
formula for the Dirichlet form:
\begin{equation}
\label{n-12}
\mc E_T (f) \;=\; \frac T4 \, \sum_{j=1}^2  \sum_{x\in\bb Z^2} 
\int \big\{ f(\sigma^{x,x+e_j}\eta)-f(\eta) \big\}^2 \, \nu_\alpha
(d\eta) \; .  
\end{equation}

Lemma \ref{n-l9} below is the main estimate of this section.  For a
continuous function $H: \bb R_+\to\bb R$ with compact support in
$(0,\infty)$, let $V_T(\eta) = V^H_T(\eta)$ be given by
\begin{align*}
V_T(\eta) \; &=\; \sum_{j=1}^2 \sum_{x\in \bb Z^2} 
\frac{x_j}{|x|^2}\, H\Big( \frac{\log |x|}{\log T}\Big) \, \big[
\eta(x+e_j) - \eta(x)\big] \\
& -\; \frac 1{\log T}\, \sum_{j=1}^2 \sum_{x\in \bb Z^2} 
\frac{x^2_j}{|x|^4}\, H\Big( \frac{\log |x|}{\log T}\Big)^2 \, \big[
\eta(x+e_j) - \eta(x)\big]^2\;.
\end{align*}
Most of the time we omit the superscript $H$ of $V^H_T(\eta)$.

\begin{lemma}
\label{n-l9}
Let $H: \bb R_+\to\bb R$ be a continuous function with compact support
in $(0,\infty)$. Then, for all $\ell\ge 1$,
\begin{equation*}
\limsup_{T\to\infty} \frac{\log T}T \log
\bb P_{\alpha} \Big[ \int _0^1 V^H_T(\eta_s) \, ds \ge \ell \Big]
\;\le\; -\, \ell\;.
\end{equation*}
\end{lemma}

\begin{proof}
By Chebyshev's exponential inequality, it is enough to prove that
\begin{equation}
\label{n-10}
\limsup_{T\to\infty} \frac{\log T}T \log
\bb E_{\alpha} \Big[ \exp\Big\{ \frac T{\log T} 
\int _0^1 V_T(\eta_s) \, ds \Big\} \Big] \;\le\; 0\;.
\end{equation}
By Feynman-Kac' formula (cf. \cite[Section A.1.7]{kl}), the left hand
side is bounded by
\begin{equation}
\label{n-13}
\limsup_{T\to\infty} \sup_f \Big\{ 
\int V_T(\eta) \, f(\eta)\, \nu_\alpha(d\eta)  
\,-\, (\log T)\,  D_T(f) \Big\} \;,
\end{equation}
where the supremum is carried over all densities $f$ and $D_T(f)$
represents the Dirichlet form of $\sqrt{f}$ defined in \eqref{n-12}:
$f\ge 0$, $\int f \, d\nu_\alpha = 1$, $D_T(f) = \mc E_T(\sqrt{f})$.

Consider the linear (in $H$) term of $V_T(\eta)$. Performing a change
of variables $\eta' = \sigma^{x,x+e_j}$ we obtain that
\begin{align*}
& \sum_{j=1}^2 \sum_{x\in \bb Z^2} 
\frac{x_j}{|x|^2}\, H\Big( \frac{\log |x|}{\log T}\Big) \, 
\int  \big[ \eta(x+e_j) - \eta(x)\big] \, f(\eta)\, \nu_\alpha(d\eta)
\\
& \; =\; -\, \frac 12\, \sum_{j=1}^2 \sum_{x\in \bb Z^2} 
\frac{x_j}{|x|^2}\, H\Big( \frac{\log |x|}{\log T}\Big) \, 
\int  \big[ \eta(x+e_j) - \eta(x)\big] \, \big[ f(\sigma^{x,x+e_j}\eta) -
f(\eta) \big]\, \nu_\alpha(d\eta)
\end{align*}
Write the difference $f(\sigma^{x,x+e_j}\eta) - f(\eta)$ as $(b-a) =
(\sqrt{b} - \sqrt{a}) (\sqrt{b} + \sqrt{a})$ and apply Young's
inequality to bound the previous expression by
\begin{align*}
& \frac 1{4\log T}\, \sum_{j=1}^2 \sum_{x\in \bb Z^2} 
\frac{x^2_j}{|x|^4}\, H\Big( \frac{\log |x|}{\log T}\Big)^2 \, 
\int  \big[ \eta(x+e_j) - \eta(x)\big]^2 \, \big[ \sqrt{f(\sigma^{x,x+e_j}\eta)} +
\sqrt{f(\eta)} \big]^2 \, \nu_\alpha(d\eta) \\
&\quad +\; \frac{\log T}4 \, \sum_{j=1}^2 \sum_{x\in \bb Z^2} 
\int  \big[ \sqrt{f(\sigma^{x,x+e_j}\eta)} -
\sqrt{f(\eta)} \big]^2 \, \nu_\alpha(d\eta)\;,
\end{align*}
By \eqref{n-12}, the second line is $(\log T)\, D_T(f)$ and cancels
with the second term in \eqref{n-13}. On the other hand, since
$(\sqrt{b} + \sqrt{a})^2 \le 2 (a+b)$, a change of variables $\eta' =
\sigma^{x,x+e_j} \eta$ yields that the first line is equal to
\begin{equation*}
\frac 1{\log T}\, \sum_{j=1}^2 \sum_{x\in \bb Z^2} 
\frac{x^2_j}{|x|^4}\, H\Big( \frac{\log |x|}{\log T}\Big)^2 \, 
\int  \big[ \eta(x+e_j) - \eta(x)\big]^2 \, f(\eta) \,
\nu_\alpha(d\eta) \;, 
\end{equation*}
which is exactly the quadratic (in $H$) term in $V_T(\eta)$. This proves
\eqref{n-10}, and therefore the lemma.
\end{proof}

For the proof of the large deviations principle, we need to restate
Lemma \ref{n-l9} in terms of the polar measure $\bar\mu^T$. For the
piece which is linear in $H$ this is just a summation by parts. For
the one which is quadratic in $H$, it relies on the superexponential
estimates presented in the previous section.

Fix a smooth function $H:\bb R_+ \to \bb R$ with compact support in
$[a,b]$, where $0<a<b$. We claim that 
\begin{equation}
\label{n-08}
-\, 2\pi \int H'(r) , \mu^{1,T}(dr) \;=\; \sum_{j=1}^2 \sum_{x\in \bb Z^2} 
\frac{x_j}{|x|^2}\, H\Big( \frac{\log |x|}{\log T}\Big) \, \big[
\eta(x+e_j) - \eta(x)\big] \; +\; R_T \;,
\end{equation}
where the absolute value of $R_T$ is bounded by $C_0 T^{-a}$ for some
finite constant $C_0$ which depends only on $H$.  This result follows
from a summation by parts on the right-hand side. The derivative of
$H$ provides the term on the left-hand side.  The divergence of
$(x_1/|x|^2, x_2/|x|^2)$ vanishes because $\log |x|$ is harmonic and
$x_j/|x|^2 = \partial_{x_j} \log |x|$.

For a continuous function $H: \bb R_+\to\bb R$ with compact support in
$(0,\infty)$, let $V_{T,\delta} (\eta) = V^H_{T,\delta} (\eta)$ be
given by
\begin{equation*}
V_{T,\delta}(\eta) \; =\; -\, 2\pi \int_{\bb R_+}  H'(r) \,
\mu^{1,T}(\psi_{r, \delta}) \, dr
-\; 4\pi \, \int_{\bb R_+}  H(r)^2 \, 
\sigma \big(m_{\delta,T} (r) \big) \, dr \;,
\end{equation*}
where
\begin{equation*}
m_{\delta,T} (r) \;=\; 
\begin{cases}
\mu^{1,T}\big( \psi_{r, \delta} \big) 
& \text{ if $r < 1/2$,} \\
\alpha & \text{ if $r \ge 1/2$.}
\end{cases}
\end{equation*}

\begin{corollary}
\label{n-l8}
Let $H: \bb R_+\to\bb R$ be a function in $C^1_K((0,\infty))$. Then,
for all $\ell\ge 1$,
\begin{equation*}
\limsup_{\delta \to 0} \limsup_{T\to\infty} \frac{\log T}T \log
\bb P_{\alpha} \Big[ \int _0^1 V^H_{T,\delta} (\eta_s) \, ds \ge
\ell \, \Big] \;\le\; -\, (\ell - 1)\;.
\end{equation*} 
\end{corollary}

\begin{proof}
By Lemma \ref{n-l9} and \eqref{n-21}, it is enough to show that
\begin{equation}
\label{n-31}
\limsup_{\delta \to 0} \limsup_{T\to\infty} \frac{\log T}T \log
\bb P_{\alpha} \Big[ \int _0^1 \Big\{ V^H_{T,\delta} (\eta_s) 
\,-\, V^H_{T} (\eta_s) \Big\}\, ds \ge 1 \, \Big] \;\le\; -\, \infty\;.
\end{equation}

The sums $V^H_{T,\delta}$, $V^H_{T}$ are expressed as a difference
between a linear term in $H$ and a quadratic term in $H$. We compare
separately the linear and the quadratic terms. By \eqref{n-08} and
Lemma \ref{n-l14}, the absolute value of the difference between the
linear terms is uniformly bounded by $1/2$ for $\delta$ small enough.

We turn to the quadratic terms.  Apply Remark \ref{n-l23} to replace
the integral $\int H(r)^2 \, \sigma \big(m_{\delta,T} (r)
\big) \, dr$ by a Riemannian sum. After this step, the difference of
the quadratic terms is seen to be equal to $W^{H^2,\delta}_T(\eta)$
introduced in \eqref{n-19}. Assertion \eqref{n-31} for the quadratic
piece follows therefore from Proposition \ref{n-l6}.
\end{proof}

The previous result rephrases Lemma \ref{n-l9} in terms of the polar
measure $\mu^{1,T}$.  We go one step further integrating in time to
express the estimate in terms of $\bar \mu^T$.  For a continuous
function $H: \bb R_+\to\bb R$ with compact support in $(0,\infty)$,
let $W^H_{T,\delta}$ be given by
\begin{equation*}
W^H_{T,\delta} \; =\; 2\pi \, \Big\{ -\, \int_{\bb R_+}  H'(r) \,
\bar \mu^{T}(\psi_{r, \delta}) \, dr
-\;2 \, \int_{\bb R_+}  H(r)^2 \, 
\sigma \big(\overline{ m}_{\delta,T} (r) \big) \, dr \Big\}\;,
\end{equation*}
where
\begin{equation*}
\overline{m}_{\delta,T} (r) \;=\; 
\begin{cases}
\bar \mu^{T}\big( \psi_{r, \delta} \big) 
& \text{ if $r < 1/2$,} \\
\alpha & \text{ if $r \ge 1/2$.}
\end{cases}
\end{equation*}
The next result follows from the previous corollary and from the
concavity of $\sigma$.

\begin{corollary}
\label{n-l15}
Let $H: \bb R_+\to\bb R$ be a function in $C^1_K((0,\infty))$. Then,
for all $\ell\ge 1$,
\begin{equation*}
\limsup_{\delta \to 0} \limsup_{T\to\infty} \frac{\log T}T \log
\bb P_{\alpha} \Big[ W^H_{T,\delta}  \ge
\ell \, \Big] \;\le\; -\, (\ell - 1)\;.
\end{equation*} 
\end{corollary}

One recognizes in $W^H_{T,\delta}$ the germ of an energy functional.
For a function $G$ in $C^1_K((0,\infty))$, let $\mc Q_{\alpha, G}: \mc
M_0 \to \bb R$ be given by
\begin{equation}
\label{n-51}
\mc Q_{\alpha, G} (m(r)\, dr) \;=\; -\, \int_0^\infty  G'(r)\, m(r) \, dr \;-\; 
2 \int_{0}^\infty \sigma (m_\alpha(r)) \, G(r)^2 \, dr\;,
\end{equation}
where 
\begin{equation}
\label{n-35}
m_\alpha  (r) \;=\; 
\begin{cases}
m(r) & \text{ if $r < 1/2$,} \\
\alpha & \text{ if $r \ge 1/2$.}
\end{cases}
\end{equation}

With this notation, Corollary \ref{n-l15} can be restated as follows.
Let $H: \bb R_+\to\bb R$ be a function in $C^1_K((0,\infty))$. Then,
for all $\ell\ge 1$,
\begin{equation}
\label{n-32}
\limsup_{\delta \to 0} \limsup_{T\to\infty} \frac{\log T}T \log
\bb P_{\alpha} \Big[ \, 2\pi \, \mc Q_{\alpha, H}(\bar
\mu^{T}_\delta) \,\ge\,
\ell \, \Big] \;\le\; -\, (\ell - 1)\;.
\end{equation} 
where, for a measure $\mu\in\mc M_c$, $\mu_\delta \in \mc M_0$
represents the absolutely continuous measure whose density $m_\delta$
is given by $\mu( \psi_{r,\delta})$: for all functions $G\in
C_K((0,\infty))$,
\begin{equation}
\label{n-25}
\int_0^\infty G(r)\, \mu_\delta (dr) \;=\; 
\int_0^\infty G(r)\, \mu\big( \psi_{r,\delta} \big) \, dr\;.
\end{equation}

\section{Energy and rate function $I_\alpha$}
\label{nsec-1}

We present in this section some properties of the large deviations
rate functional. 

Fix $0<\alpha<1$.  Denote by $\mc Q_\alpha : \mc M_0 \to \bb R_+$ the
energy functional given by
\begin{equation}
\label{n-30}
\mc Q_\alpha (\mu) \;=\; \sup_{G\in C^1_K((0,\infty))}
\mc Q_{\alpha, G} (\mu) \;,
\end{equation}
where $\mc Q_{\alpha, G}$ is defined by \eqref{n-51}. Next result is
Lemma 4.1 in \cite{blm}. 

\begin{lemma}
\label{n-l16}
The functional $\mc Q_\alpha : \mc M_0 \to \bb R_+$ is convex and
lower-semicontinuous.  Moreover, if $\mc Q_\alpha (m(r) dr) < \infty$,
then $m(r)$ has a generalized derivative, denoted by $m'(r)$, and
\begin{equation}
\label{n-34}
\mc Q_\alpha (m(r) dr) \;=\; \frac{1}{8} 
\int_0^{\infty} \frac{m'(r)^2}{\sigma( m_\alpha(r))} \, dr\;. 
\end{equation}
\end{lemma}

Let $C^2(\bb R_+, \alpha)$ be the space of twice continuously
differentiable functions $\gamma\,:\,[0, \infty)\,\to\,(0,1)$ such
that $\gamma'$ has a compact support in $(0, 1/2)$ and such that
$\gamma(r) = \alpha$ for $r$ sufficiently large.  There exists
therefore $0<\beta<1$ and $0<\varepsilon<1/4$ such that
$\gamma(r)=\beta$ for $r \le \varepsilon$, and $\gamma(r)=\alpha$ for
$r\ge 1/2-\varepsilon$, $\varepsilon \le \gamma(s) \le 1- \varepsilon$
for all $s\ge 0$.  For each $\gamma$ in $C^2(\bb R_+, \alpha)$, let
$\Gamma=\Gamma_{\gamma, \alpha}: \bb R_+ \to \bb R$ be given by
\begin{equation}
\label{n-02}
\Gamma (u) \;=\; \frac 12 \Big\{ \log \frac {\gamma (u)}
{1-\gamma (u)} \,-\, 
\log \frac {\alpha}{1-\alpha} \Big\} \; \cdot
\end{equation}
Note that the space $\{ \Gamma_{\gamma, \alpha}' ,\, \gamma \in
C^2(\bb R_+, \alpha)\}$ corresponds to the space $C^1_K((0,1/2))$.

Fix $\gamma \in C^2(\bb R_+, \alpha)$, and let $J_{\gamma} : \mc M_0
\to \bb R$ be the rate-functional given by
\begin{equation}
\label{n-33}
J_{\gamma} (m(r)\, dr) \; =\;
-\, \pi \int_0^{\infty} \Gamma''(r) \, m(r)\, dr 
\; - \; \pi \int_0^{\infty} \sigma(m(r))\, \Gamma'(r)^2 
\, dr \;.
\end{equation}

Recall from \eqref{n-52} that we denote by $\mc M_{0,\alpha}$ the
space of absolutely continous measures whose density is equal to
$\alpha$ on $(1/2, \infty)$. Denote by $\mc M^{\mc Q}_{0,\alpha}$ the
set of measures in $\mc M_{0,\alpha}$ with finite energy:
\begin{equation}
\label{n-63}
\mc M^{\mc Q}_{0,\alpha} \;=\; \big\{\mu(dr) = m(r)\, dr \in \mc M_{0,\alpha} : 
\mc Q_\alpha(\mu) < \infty \big\}\;.
\end{equation}
Let $J^{\mc Q}_{\gamma} : \mc M_c\to \bb R \cup \{+\infty\}$ the
functional given by
\begin{equation}
\label{n-53}
J^{\mc Q}_{\gamma} (\mu) \;=\; 
\left\{
\begin{array}{ll}
J_{\gamma} (\mu) & \text{if $\mu \in \mc
  M^{\mc Q}_{0,\alpha}\,$,} \\
+\infty & \text{otherwise,}
\end{array}
\right.
\end{equation}
Note that $J^{\mc Q}_{\gamma}$ is defined on $\mc M_c$, while
$J_{\gamma}$ is only defined on $\mc M_0$.

Let $J^{\mc Q}: \mc M_c \to \bb R_+$ be given by
\begin{equation}
\label{n-54}
J^{\mc Q} (\mu) \;=\; \sup_{\gamma \in C^2(\bb R_+, \alpha)}  
J^{\mc Q}_{\gamma} (\mu) \;.
\end{equation}
Since the set $\{\Gamma_{\gamma,\alpha}' : \gamma \in C^2(\bb R_+,
\alpha)\}$ coincides with the set $C^1_K((0,1/2))$, on the set $\mc
M^{\mc Q}_{0,\alpha}$ the functional $J^{\mc Q}$ can be rewritten as
\begin{equation*}
J^{\mc Q} (m(r) dr) \;=\; \pi \, \sup_{H \in C^1_K((0,1/2))}  \Big\{
-\, \int_0^{\infty} H'(r) \, m(r)\, dr 
\; - \; \int_0^{\infty} \sigma(m(r))\, H(r)^2 
\, dr \Big\} \;.
\end{equation*}
The proof of Lemma \ref{n-l16} yields that if $J^{\mc Q} (m(r)
dr)<\infty$, then $m$ has a generalized derivative in $(0,1/2)$,
denoted by $m'$, and
\begin{equation}
\label{n-36}
J^{\mc Q} (m(r) dr) \;=\; \frac{\pi}{4}\, \int_0^{1/2} 
\frac{[m'(r)]^2}{\sigma(m(r))} \, dr\;.
\end{equation}

The next results asserts that in the definition of the rate function
$J^{\mc Q}$, we can replace the set $C^1_K((0,1/2))$ by the larger one
$C^1_K((0,\infty))$.

\begin{lemma}
\label{n-l17}
For $\mu \in \mc M^{\mc Q}_{0,\alpha}$, 
\begin{equation}
\label{n-55}
J^{\mc Q} (m(r) dr) \;=\; \pi\, \sup_{H \in C^1_K((0,\infty))}  \Big\{
-\, \int_0^{\infty} H'(r) \, m(r)\, dr 
\; - \; \int_0^{\infty} \sigma(m(r))\, H(r)^2 
\, dr \Big\} \;.
\end{equation}
In particular, $J^{\mc Q} = I_{\mc Q,\alpha}$ on $\mc M_c$.
\end{lemma}

\begin{proof}
Denote by $Q (m(r) dr)$ the right hand side of \eqref{n-55}. It is
clear that $J^{\mc Q} (\mu) \le Q(\mu)$ for all $\mu\in \mc M_0$. We
prove the reverse inequality for measures in $\mc M^{\mc
  Q}_{0,\alpha}$.

Fix $\mu(dr) = m(r) dr \in \mc M^{\mc Q}_{0,\alpha}$. We claim that $Q
(\mu) <\infty$. Indeed, recall the definition of $\mc Q_\alpha$
introduced in \eqref{n-30}. In formula \eqref{n-51}, take $G/2$ in
place of $G$, to obtain that
\begin{equation*}
-\, \int_0^{\infty} G'(r) \, m(r)\, dr 
\; - \; \int_0^{\infty} \sigma(m_\alpha(r))\, G(r)^2 
\, dr  \;\le\; 2\, \mc Q_{\alpha, G/2} (\mu) 
\;\le\; 2\, \mc Q_\alpha  (\mu)
\end{equation*}
for all $G \in C^1_K((0,\infty))$.  Since $m(r)=\alpha$ for $r\ge
1/2$, in the variational formula which defines $Q$, we may replace
$\sigma(m(r))$ by $\sigma(m_\alpha(r))$. After this replacement,
optimizing over $G$ yields that $Q (\mu) \le 2\, \mc Q_\alpha
(\mu)$. As $\mu$ belongs to $\mc M^{\mc Q}_{0,\alpha}$, $\mc Q_\alpha
(\mu)<\infty$ so that $Q (\mu)<\infty$, as claimed.

By Lemma \ref{n-l16}, since $Q(\mu) < \infty$, $m$ has a generalized
derivative, denoted by $m'$, and
\begin{equation*}
Q (\mu) \;=\; \;=\; \frac{\pi}{4}\, \int_0^{\infty} 
\frac{[m'(r)]^2}{\sigma(m(r))} \, dr\;.
\end{equation*}
Since $m(r)=\alpha$ for $r\ge 1/2$, $m'(r)=0$ a.s. on $[1/2, \infty)$,
and the range of the previous integral can be reduced to $[0,1/2]$,
which proves that $Q (\mu) = J^{\mc Q} (\mu)$ in view of \eqref{n-36}.

To prove the second assertion of the lemma, observe first that both
functionals coincide on the set $\mc M^{\mc Q}_{0,\alpha}$. Indeed,
the right hand side of \eqref{n-55} is just $\pi \mc Q$, where $\mc Q$
has been introduced in \eqref{n-39}, and $I_{\mc Q,\alpha}$ is equal
to $\pi \mc Q$ on $\mc M^{\mc Q}_{0,\alpha}$. It remains to show that
$I_{\mc Q,\alpha} = J^{\mc Q} = \infty$ on $[\mc M^{\mc
  Q}_{0,\alpha}]^c$. For $J^{\mc Q}$ this follows by definition. For
$I_{\mc Q,\alpha}$ the identity holds on $[\mc M_{0,\alpha}]^c$ by
definition. On the set $\mc M_{0,\alpha} \setminus \mc M^{\mc
  Q}_{0,\alpha}$, $I_{\mc Q,\alpha}(\mu) = \pi \mc Q(\mu) = \infty$. 
\end{proof}

\section{The upper bound}
\label{sec3}

The proof of the upper bound is similar to the one presented in
\cite{cll}, but relies on the energy estimate proved in the previous
section to restrict the set of measures to the ones with finite energy
on $\bb R_+$.

We follow \cite{blm} with a minor improvement. Instead of considering
$\bar\mu^T$ as a density function on $\bb R_+$ we defined here
$\bar\mu^T$ as a measure with mass points. This is more natural, but
creates an extra minor difficulty, as we have to show that at the level
of the large deviations, we may exclude measures which are not
absolutely continuous.

The proof of the large deviations principle is based on the following
perturbations of the dynamics. Fix $\gamma\in C^2(\bb R_+,\alpha)$ and
recall from \eqref{n-02} the definition of the function $\Gamma: \bb
R_+ \to \bb R$ introduced in. Let $\Gamma_T(x)$, $T>0$, $x\in \bb
Z^2$, be given by
\begin{equation*}
\Gamma_T(x) \;=\; \Gamma (\sigma_T(x))\;.
\end{equation*}
Denote by $L_{T,\gamma}$ the generator of the inhomogeneous exclusion
process in which a particle jumps from $x$ to $y$ at rate
$\exp\{\Gamma_T (y)- \Gamma_T (x)\}$:
\begin{equation*}
(L_{T,\gamma} f)(\eta) \;=\; \frac T2 \sum_{x\in\bb Z^2} 
\sum_{y:|x-y|=1} \eta(x) \{1-\eta (y)\}\,
e^{\Gamma_T (y)- \Gamma_T (x)} \,
[f(\sigma^{x,y} \eta) - f(\eta)]\; .  
\end{equation*}
Denote by $\nu_{T, \gamma}$ the product measure on $\{0,1\}^{\bb
  Z^2}$, with marginals given by
\begin{equation}
\label{n-56}
\nu_{T, \gamma} \{\eta, \, \eta(x)=1\} \; =\;
\gamma(\sigma_T(x)) \; .  
\end{equation}
The measure $\nu_{T, \gamma}$ coincides with $\nu_\alpha$ outside a
ball of radius $T^{1/2 - \varepsilon}$ centered at the origin, for
some $\varepsilon >0$.  Moreover, a simple computation shows that
$\nu_{T,\gamma}$ is an invariant reversible measure for the Markov
process with generator $L_{T,\gamma}$.  Denote by $\bb P_{T, \gamma}$
the probability measure on $D(\bb R_+, \{0,1\}^{\bb Z^2})$ induced by
Markov process whose generator is $L_{T,\gamma}$ and which starts from
$\nu_{T,\gamma}$.

This section is organized as follows. We first define four subsets of
measures whose complements have superexponentially small
probabilities. Then, we show that on these sets a family of
martingales can be expressed in terms of the polar measure
$\bar\mu^T$. These explicit formulae and a min-max argument due to
Varadhan permit to conclude the proof of the upper bound.

\smallskip\noindent{\bf A. Polar measure at $[1/2,\infty)$.}  Let
$\{G_m : m\ge 1\}$ be a sequence of functions in $C_K((1/2,\infty))$
which is dense with respect to the supremum norm. For $\kappa >0$ and
$n\ge 1$, denote by $A_{n,\kappa}$ the closed subspace of $\mc M_c$
defined by
\begin{equation}
\label{n-38}
A_{n,\kappa} \; =\; \Big\{ \mu\in \mc M_c :\,
\Big | \int G_m(r)  \, \mu(dr) \, -\,  \alpha \int G_m(r) \, dr \, \Big| 
\,\le\, \kappa \text{ for $1\le m\le n$ } \Big\}\; .  
\end{equation}
By Lemma \ref{n-l1} and \eqref{n-21}, for every $\kappa>0$ and $n\ge
1$,
\begin{equation}
\label{n-28}
\limsup_{T\to\infty} \frac{\log T} T \log \bb P_{\alpha} \big[ \,
\bar \mu^T \not \in A_{n,\kappa} \, \big] \; =\; -\, \infty\;.
\end{equation}

\smallskip\noindent{\bf B. Energy functionals.}  Recall the definition
of the functionals $\mc Q_{\alpha, H}$ defined by \eqref{n-51}. Fix a
sequence $\{H_p : p\ge 1\}$ of smooth functions, $H_p\in C^2_K(\bb
R_+)$, dense in $C^1_K(\bb R_+)$.  For $q\ge 1$, $\ell>0$, let
$B_{q,\ell}$ be the set of paths with truncated energy bounded by
$\ell$:
\begin{equation}
\label{n-26}
B_{q,\ell} \;=\; \big\{ \mu\in\mc M_0 : 2\pi\, 
\max_{1\le p\le q} \mc Q_{\alpha, H_p}(\mu) \le \ell\big\} \;.
\end{equation}
By \eqref{n-21} and \eqref{n-32}, for any $q\ge 1$ and $\ell >0$
\begin{equation}
\label{n-22}
\limsup_{\delta \to 0} \limsup_{T\to\infty} \frac {\log T}T  \log
\bb P_{\alpha} \big [ \,
\bar \mu^T_\delta \not\in B_{q,\ell} \, \big] \; \le \; -(\ell -1) \;.
\end{equation}

\smallskip\noindent{\bf C. Absolutely continuous measures.}  Let
$\{F_i : i\ge 1\}$ be a sequence of nonnegative functions in
$C_K((0,1/2))$ which is dense with respect to the supremum norm in the
space of nonnegative functions in $C_K((0,1/2))$. For $\kappa >0$ and
$m\ge 1$, denote by $C_{m,\kappa}$ the closed subspace of $\mc M_c$
defined by
\begin{equation}
\label{n-43}
C_{m,\kappa} \; =\; \Big\{ \mu\in \mc M_c :
\int F_i(r)  \, \mu(dr) \,\le\, 
\int F_i (r) \, dr \,+\, \kappa \text{ for $1\le i\le m$ } \Big\}\; .  
\end{equation}
By \eqref{n-41}, for every $\kappa>0$ and $m\ge 1$, 
\begin{equation}
\label{n-42}
\limsup_{T\to\infty} \frac{\log T} T \log \bb P_{\alpha} \big[ \,
\bar \mu^T \not \in C_{m,\kappa} \, \big] \; =\; -\, \infty\;.
\end{equation}

\smallskip\noindent{\bf D. Ergodic bounds.}  Fix $\gamma \in C^2(\bb R_+,
\alpha)$, $\delta >0$ and $\varrho >0$. Recall from \eqref{n-27} the
definition of the local function $\Psi_{\delta, T}(j, x,\eta)$.  Let
$B^{\delta, \varrho}_{T, \gamma}$ be the set defined by
\begin{equation}
\label{n-04}
B^{\delta, \varrho}_{T, \gamma} \; =\; \Big \{\eta:
\int_0^1 W^{\delta}_{T,\gamma} (\eta_s) \, ds\, \le \varrho
\Big \} \; ,
\end{equation}
where
\begin{gather*}
W^{\delta}_{T,\gamma} (\eta) \; =\; \frac 1{2 \log T}
\sum_{x\in \bb Z^2_*}  \sum_{j=1}^2 \frac {x_j^2}{|x|^4}
\, \Gamma'(\sigma_T(x))^2\, \Psi_{\delta, T}(j, x,\eta) \;.
\end{gather*}
As the support of $\Gamma'$ is contained in $(0,1/2)$, by Lemma
\ref{n-l12}, for every $\gamma\in C^2(\bb R_+, \alpha)$, $\varrho>0$
\begin{equation}
\label{n-l13}
\limsup_{\delta \to 0} \limsup_{T\to\infty} \frac {\log T}T 
\, \log \bb P_{\alpha} \big[ \, 
\big( B^{\delta, \varrho}_{T, \gamma} \big)^c \, \big] \; =\; - \infty\;.   
\end{equation}

\smallskip\noindent{\bf E. Radon-Nikodym derivatives.}  Fix $\gamma
\in C^2(\bb R_+, \alpha)$. Recall from the paragraph below
\eqref{n-56} the definition of the measure $\bb P_{T, \gamma}$, and
denote by $d\bb P_\alpha/d \bb P_{T, \gamma}$ the Radon-Nikodym
derivative of the measure $\bb P_\alpha$ with respect to the measure
$\bb P_{T, \gamma}$ restricted to the $\sigma$-algebra generated by
$\eta_s$, $0\le s\le 1$.

The Radon-Nikodym derivative $d \bb P_{T, \gamma} /d \bb P_\alpha$ can
be written as the product of three exponentials:
\begin{equation}
\label{n-58}
\frac{d \bb P_{T, \gamma}}{d \bb P_\alpha} \;=\; \Psi_{\rm stat} \,
\Psi_{\rm pot} \, \Psi_{\rm dyn} \;.
\end{equation}
The first exponential corresponds to the Radon-Nikodym derivative of
the initial states: $d\nu_{T, \gamma}/d\nu_\alpha$:
\begin{equation*}
\Psi_{\rm stat} \;=\; \exp \sum_{x\in \bb Z^2_*}
\Big\{ \eta_0 (x) \log \Big ( \frac{\gamma (\sigma_T(x))}
\alpha \Big ) \; +\;  [1-\eta_0(x)]  \log
\Big ( \frac{1-\gamma(\sigma_T(x))}  {1-\alpha}\Big) \Big\}  \;.
\end{equation*}
The second one is associated to the potential $V(\eta)= \sum_{x\in \bb
  Z^2_*} \Gamma_T(x) \, \eta(x)$:
\begin{equation*}
\Psi_{\rm pot} \;=\; \exp \Big\{ \sum_{x\in \bb Z^2_*} \Gamma_T(x) 
\, \{ \eta_1(x) - \eta_0(x) \} \Big\}\; .
\end{equation*}
The last one is the exponential corrector which turns $e^{V(\eta_s)}$
a martingale: $\Psi_{\rm dyn} = \exp\{ - \int_0^1 e^{-V(\eta_s)} \,
L_T e^{V(\eta_s)} \, ds\}$, so that
\begin{align*}
\Psi_{\rm dyn} \;=\; \exp \Big\{ -\; \frac T2 \int_0^t \sum_{x\in\bb
  Z^2} \sum_{y: |y-x|=1} \eta_s(x) \, [1-\eta_s(y)]\,
\{ e^{\Gamma_T(y) - \Gamma_T(x)} -1 \} \, ds\,  \Big\}\;.
\end{align*}
where $\Gamma_T(z)$ has been defined above \eqref{n-56}. 

Assume that the support of $\gamma'$ is contained in $[a,b]\subset
(0,1/2)$. In this case, $|\log \Psi_{\rm stat}|$ and $|\log \Psi_{\rm
  pot}|$ are bounded by $C_0 T^{2b} \ll T$. On the other hand, by a
Taylor's expansion and the harmonicity of $\log |x|$ in $\bb R^2$,
\begin{equation}
\label{n-60}
\frac {\log T}{T} \log \Psi_{\rm dyn} \;=\;
- \, \pi  \, \int \Gamma''(r) \, \bar\mu^T(dr) \;
-\; \int_0^1 W_\gamma (\eta_s) \, ds \; +\; o_T(1)\;,
\end{equation}
where 
\begin{equation}
\label{n-59}
W_\gamma (\eta) \;=\; \frac 1{4\, \log T} \sum_{j=1}^2
\sum_{x\in \bb Z} [\Gamma' (\sigma_T(x))]^2 \, \frac
{x^2_j}{|x|^4}  \,  \big[\eta(x+e_j) - \eta(x)\big]^2\;,
\end{equation}
and $\lim_T o_T(1)=0$.

It follows from the previous estimates that there exists a finite
constant $C_0$ depending only on $\gamma$ such that
\begin{equation}
\label{n-57}
\Big| \, \log \frac{d \bb P_{T, \gamma}}{d \bb P_\alpha} \, \Big|\;\le\;
C_0 \, \frac{T}{\log T}\;\cdot
\end{equation}

Recall from \eqref{n-33} the definition of the functional $J_\gamma$,
and from \eqref{n-25} the definition of the measure $\mu_\delta$.
Next result follows from the estimates of $\Psi_{\rm stat}$,
$\Psi_{\rm pot}$, \eqref{n-60} and \eqref{n-29} to replace
$\phi_{r,\delta}$ by $\psi_{r,\delta}$.

\begin{lemma}
\label{n-l3}
Fix $\gamma \in C^2(\bb R_+, \alpha)$, $0<\delta \le \varrho$.  There
exists a finite constant $C_0$, depending only on $\gamma$, such that
on the set $B^{\delta ,\varrho}_{T, \gamma}$ introduced in \eqref{n-04},
\begin{align*}
\log \frac {d\bb P_\alpha}{d \bb P_{T, \gamma}}
\le\; -\, \frac {T} {\log T} \, J_{\gamma} (\bar\mu^T_\delta)
\;+\; C_0\, \varrho \, \, \frac T{\log T} \;\cdot
\end{align*}
\end{lemma}

\begin{remark}
\label{n-l11}
In the proof of the large deviations upper bound, the pieces
$\Psi_{\rm stat}$ and $\Psi_{\rm pot}$ of the Radon-Nikodym derivative
$d \bb P_{T, \gamma}/d\bb P_\alpha$ are the ones which forbid
perturbations $\gamma$ which are not constant outside a compact subset
of $(0,1/2)$. Indeed, if the support of $\gamma'$ has a nonempty
intersection with $(1/2,\infty)$ $\Psi_{\rm stat}$ and $\Psi_{\rm
  pot}$ are of an order much larger than $\exp\{T/\log T\}$ because of
the volume of the region $\{x\in \bb Z^2 : T^a \le |x| \le T^b\}$ for
$a\ge 1/2$.

This is not the case of $\Psi_{\rm dyn}$ due to the presence of the
factor $|x|^{-2}$. Indeed, as shown in the proof of Proposition
\ref{n-l6}, to estimate $\Psi_{\rm dyn}$ in the case of a perturbation
$\gamma$ which is not constant outside a compact subset of $(0,1/2)$,
we may divide $\bb Z^2$ in three regions $\bb R^{(1)}_{T,\kappa}$,
$\bb R^{(2)}_{T,\kappa}$ and $\bb R^{(3)}_{T,\kappa}$. All terms in
the first region belong to the set $\{x\in \bb Z^2 : |x| \le
T^{1/2-\kappa} \}$ and can be handled as in \cite{cll}. The sum over
$\bb R^{(2)}_{T,\kappa}$ is negligible if $\kappa$ is small
(cf. equation \eqref{n-18}), while the sum over $\bb
R^{(3)}_{T,\kappa}$ is fixed, as proved in Lemma \ref{n-l2}.  

This explains why we are able to prove an energy estimate on $\bb R_+$
and not just on $(0,1/2)$: the expression which appears in the proof
of the energy estimate stated in Lemma \ref{n-l9} is similar to
$\Psi_{\rm dyn}$ and there are no terms corresponding to $\Psi_{\rm
  stat}$ and $\Psi_{\rm pot}$.
\end{remark}

\smallskip\noindent{\bf F. Proof of the upper bound.}  We are now in a
position to prove the upper bound. Fix $\gamma\in C^2(\bb R_+,
\alpha)$, and let $\Gamma$ be the function associated to $\gamma$ by
\eqref{n-02}.  Fix $\varrho>0$, $\delta>0$, $\varepsilon>0$,
$\kappa_1>0$, $\kappa_2>0$, $q\ge 1$, $n\ge 1$, $m\ge 1$, $\ell \ge
1$, and recall the definition of the sets $A_{n,\kappa}$,
$B_{q,\ell}$, $C_{m, \kappa}$, $B^{\delta ,\varrho}_{T, \gamma}$
introduced in \eqref{n-38}, \eqref{n-26}, \eqref{n-43} and
\eqref{n-04}.  Let $B^{\varepsilon,T}_{q,\ell} = \{\bar
\mu^T_\varepsilon \in B_{q,\ell}\}$.  It follows from \eqref{n-21},
\eqref{n-28}, \eqref{n-22}, \eqref{n-42} and \eqref{n-l13} that for
any subset $A$ of $\mc M_c$,
\begin{align*}
& \limsup_{T\to\infty} \frac{\log T} T \log \bb P_\alpha
[ \bar \mu^T \in A] \\
&\quad \le \; \max\Big\{ \limsup_{T\to\infty}
\frac {\log T} T \log\bb P_\alpha \Big[ \bar \mu^T \in A \cap
A_{n,\kappa_1} \cap C_{m,\kappa_2} \,,\, 
B^{\delta ,\varrho}_{T, \gamma} \cap B^{\varepsilon,T}_{q,\ell}  \Big] \, , \,  
R_\gamma \Big\} \;,
\end{align*}
where $R_\gamma = \max\{ C_\gamma (\delta , \varrho) \,,\, C
(\varepsilon, q, \ell)\}$ and
\begin{equation}
\label{n-44}
\limsup_{\delta \to 0} C_\gamma (\delta ,\varrho) \;=\;  -\infty
\quad\text{and}\quad \limsup_{\varepsilon\to 0} C (\varepsilon , q, \ell)
\;\le\; -(\ell-1)
\end{equation}
for all $\varrho>0$, $q\ge 1$, $\ell\ge 1$, $\gamma\in C^2(\bb R_+,
\alpha)$.

To estimate the right hand side of the penultimate formula, observe
that
\begin{equation*}
\bb P_\alpha [ \bar \mu^T \in D \, , \,
B^{\delta ,\varrho}_{T, \gamma} \cap B^{\varepsilon,T}_{q,\ell} ] \; =\;
\bb E_{T, \gamma} \Big[ \frac {d\bb P_\alpha}{d\bb P_{T, \gamma}}
\mb 1 \big\{\bar \mu^T \in D \,,\,
B^{\delta ,\varrho}_{T, \gamma} \cap B^{\varepsilon, T}_{q,\ell} \big\} \Big]\; ,
\end{equation*}
where $\bb E_{T, \gamma}$ represents the expectation with respect to
$\bb P_{T, \gamma}$, and $D = A \cap A_{n,\kappa_1}\cap C_{m,\kappa_2}$.
By Lemma \ref{n-l3}, on the set $B^{\delta ,\varrho}_{T, \gamma}$, if
$\delta \le \varrho$,
\begin{equation*}
\label{eq:6}
\log \frac { d\bb P_\alpha}{d \bb P_{T, \gamma}}
\;\le\;  -\, \frac T{\log T} \, J_{\gamma, \delta} (\bar\mu^T) 
\; +\; C(\gamma)\, \varrho \, \frac{T}{\log T} \; \cdot
\end{equation*}
where $J_{\gamma, \delta} : \mc M_c \to \bb R$ is the functional given
by
\begin{equation*}
J_{\gamma, \delta} (\mu) \;=\; J_{\gamma} (\mu_\delta)\;.
\end{equation*}

On the set $\{\bar \mu^T \in A_{n,\kappa_1} \cap C_{m,\kappa_2} \}
\cap B^{\varepsilon, T}_{q,\ell}$, we may replace the functional
$J_{\gamma, \delta} (\bar\mu^T)$ by $J^{\varepsilon,
  q,\ell,n,\kappa_1, m,\kappa_2}_{\gamma, \delta} (\bar\mu^T)$, where
\begin{equation*}
J^{\varepsilon, q,\ell,n,\kappa_1, m,\kappa_2}_{\gamma, \delta} (\mu) \;=\;
\begin{cases}
J_{\gamma, \delta} (\mu) & \text{if $\mu \in A_{n,\kappa_1} \cap C_{m,\kappa_2}$
  and $\mu_\varepsilon \in B_{q,\ell}\, $,} \\
+ \infty & \text{otherwise.}
\end{cases}
\end{equation*}
To avoid long formulas, write $J^{\varepsilon, q,\ell,n,\kappa_1,
  m,\kappa_2}_{\gamma, \delta}$ as $J^{\star}_{\gamma, \delta}$. Note
that $J^{\star}_{\gamma, \delta}$ is lower semi-continuous because the
set $A_{n,\kappa_1} \cap C_{m,\kappa_2} \cap \{\mu :
\mu_\varepsilon \in B_{q,\ell}\}$ is closed.

Up to this point, we proved that for all $\delta\le\varrho$,
\begin{align*}
& \limsup_{T\to\infty} \frac{\log T} T \log \bb P_\alpha
[ \bar \mu^T \in A] \\
&\qquad \le \; \max \Big\{ 
\sup_{\mu \in A}  - J^{\star}_{\gamma, \delta} (\mu) 
\,+\, C (\gamma)\, \varrho \, , \,  
C_\gamma (\delta , \varrho) \,,\, C (\varepsilon
, q, \ell) \Big\} \;.
\end{align*}
where $C(\gamma)$ is a finite constant which depends only on $\gamma$,
while the other terms satisfy \eqref{n-44}.

Optimize the previous inequality with respect to all parameters and
assume that the set $A$ is closed (and therefore compact because so is
$\mc M_c$).  Since, for each fixed set of parameters, the functional
$J^{\star}_{\gamma, \delta}$ is lower semi-continuous, we may apply
the arguments presented in \cite[Lemma A2.3.3]{kl} to exchange the
supremum with the infimum.  In this way we obtain that the last
expression is bounded above by
\begin{equation*}
\sup_{\mu \in A} 
\inf_{\substack{\varepsilon, q,\ell,n,\kappa_1, m,\kappa_2\\
 \gamma, \delta \le \varrho}} 
\max \Big\{ -   J^{\star}_{\gamma, \delta} (\mu) 
\,+\, C (\gamma)\, \varrho \, , \,  
C_\gamma (\delta , \varrho) \,,\, C (\varepsilon
, q, \ell) \Big\} \;.
\end{equation*}

Fix $\mu \in \mc M_c$, and let $n \uparrow\infty$, and
$\kappa_1\downarrow 0$, and then $m \uparrow\infty$, and
$\kappa_2\downarrow 0$ in $J^{\star}_{\gamma, \delta} (\mu)$. Keep in
mind that $\mu$ is fixed as well as $J_{\gamma, \delta} (\mu)$, the
only object which is changing with the variables $n$, $\kappa_1$, $m$
and $\kappa_2$ is the set at which $J^{\star}_{\gamma, \delta}$ takes
the value $+\infty$. Use the closeness of the sets $A_{n,\kappa_1}$,
$C_{m,\kappa_2}$ to conclude that the previous expression is bounded
by
\begin{equation*}
\sup_{\mu \in A} 
\inf_{\substack{\varepsilon, q,\ell \\
 \gamma, \delta \le \varrho}} 
\max \Big\{ -   J^{\varepsilon, q,\ell}_{\gamma, \delta} (\mu) 
\,+\, C (\gamma)\, \varrho \, , \,  
C_\gamma (\delta , \varrho) \,,\, C (\varepsilon
, q, \ell) \Big\} \;,
\end{equation*}
where
\begin{equation*}
J^{\varepsilon, q,\ell}_{\gamma, \delta} (\mu) \;=\; 
\left\{
\begin{array}{ll}
 J_{\gamma, \delta} (\mu) & \text{if $\mu \in \mc M_{0,\alpha}$ and 
$\mu_\varepsilon \in B_{q,\ell}\,$,} \\
+\infty & \text{otherwise,}
\end{array}
\right.
\end{equation*}
and $\mc M_{0,\alpha}$ is the set introduced just below
\eqref{n-39}. 

Let now $\varepsilon\downarrow 0$. We claim that for all $\mu \in \mc M_c$,
\begin{equation}
\label{n-45}
J^{q,\ell}_{\gamma, \delta} (\mu) \;\le\;
\liminf_{\varepsilon\to 0} J^{\varepsilon, q,\ell}_{\gamma, \delta}
(\mu)\;, 
\end{equation}
where
\begin{equation}
\label{n-46}
J^{q,\ell}_{\gamma, \delta} (\mu) \;=\; 
\left\{
\begin{array}{ll}
J_{\gamma, \delta} (\mu) & \text{if $\mu \in \mc M_{0,\alpha} \cap B_{q,\ell}\,$,} \\
+\infty & \text{otherwise,}
\end{array}
\right.
\end{equation}
Indeed, fix $\mu\in \mc M_c$. We may assume that $\mu \in \mc
M_{0,\alpha}$, otherwise $J^{\varepsilon, q,\ell}_{\gamma, \delta}
(\mu) = J^{q,\ell}_{\gamma, \delta} (\mu) = \infty$ for all
$\varepsilon>0$. Note that $\mu_\varepsilon \to \mu$ as
$\varepsilon\downarrow 0$. Since $B_{q,\ell}$ is a closed set, if
$\mu\not\in B_{q,\ell}$, $\mu_\varepsilon\not\in B_{q,\ell}$ for
$\varepsilon$ small enough and both sides of \eqref{n-45} are equal to
$+\infty$. It remains to consider the case $\mu\in B_{q,\ell}$. Here,
by definition, $J^{q,\ell}_{\gamma, \delta} (\mu) = J_{\gamma, \delta}
(\mu) \le J^{\varepsilon, q,\ell}_{\gamma, \delta} (\mu)$ for all
$\varepsilon>0$, which proves claim \eqref{n-45}.

In view of the second bound in \eqref{n-44}, up to this point we
proved that for all closed subset $A$ of $\mc M_c$,
\begin{align*}
& \limsup_{T\to\infty} \frac{\log T} T \log \bb P_\alpha
[ \bar \mu^T \in A] \\
&\qquad \le \; - \, \inf_{\mu \in A} 
\sup_{q,\ell, \gamma, \delta \le \varrho}
\min \Big\{ 
J^{q,\ell}_{\gamma, \delta} (\mu) 
\,-\, C (\gamma)\, \varrho \, , \,  
- C_\gamma (\delta , \varrho) \,,\, \ell - 1 \Big\} \;,
\end{align*}
where $J^{q,\ell}_{\gamma, \delta}$ is given by \eqref{n-46}. We claim
that for all $\mu\in \mc M_c$,
\begin{equation}
\label{n-47}
J^{\ell}_{\gamma, \delta} (\mu) \;\le\;
\sup_{q} J^{q,\ell}_{\gamma, \delta} (\mu) \;,
\end{equation}
where
\begin{equation*}
J^{\ell}_{\gamma, \delta} (\mu) \;=\; 
\left\{
\begin{array}{ll}
J_{\gamma, \delta} (\mu) & \text{if $\mu \in \mc
  M_{0,\alpha}$ and $\mc Q_{\alpha}(\mu) \le \ell \,$,} \\
+\infty & \text{otherwise.}
\end{array}
\right.
\end{equation*}
Indeed, suppose first that $\mc Q_{\alpha}(\mu) > \ell$. In this case,
since $H_p$ is a dense sequence, for all $q$ sufficiently large, 
$\max_{1\le p\le 1} \mc Q_{\alpha, H_p}(\mu) > \ell$ so that both
sides of \eqref{n-47} are equal to $+\infty$. On the other hand, if
$\mc Q_{\alpha}(\mu) \le \ell$ both sides are equal to $J_{\gamma,
  \delta} (\mu)$. This proves the claim.

We now assert that for all $\mu\in \mc M_c$,
\begin{equation}
\label{n-48}
J^{\mc Q}_{\gamma, \delta} (\mu) \;\le\;
\sup_{\ell} \min \Big\{ 
J^{\ell}_{\gamma, \delta} (\mu) \, , \,  
\ell - 1 \Big\}  \;,
\end{equation}
where
\begin{equation}
\label{n-49}
J^{\mc Q}_{\gamma, \delta} (\mu) \;=\; 
\left\{
\begin{array}{ll}
J_{\gamma, \delta} (\mu) & \text{if $\mu \in \mc
  M_{0,\alpha}$ and $\mc Q_{\alpha}(\mu) < \infty \,$,} \\
+\infty & \text{otherwise,}
\end{array}
\right.
\end{equation}
Indeed, if $J^{\ell}_{\gamma, \delta} (\mu) = \infty$ for all $\ell\ge
1$, there is nothing to prove. If this is note the case, by definition
of $J^{\ell}_{\gamma, \delta}$, $\mc Q_\alpha(\mu) \le m$ for some
$m\ge 1$, and $J^{\ell}_{\gamma, \delta} (\mu) = J_{\gamma, \delta}
(\mu) = J^{\mc Q}_{\gamma, \delta} (\mu)$. This proves
\eqref{n-48}. Recall from \eqref{n-63} that we denote by $\mc
M^{\mc Q}_{0,\alpha}$ the set of measures $\mu$ in $\mc M_c$ such that $\mu
\in \mc M_{0,\alpha}$ and $\mc Q_{\alpha}(\mu) < \infty$, which is the
set appearing in the definition of $J^{\mc Q}_{\gamma, \delta}$.

Putting together the previous two estimates we conclude that for all
closed subset $A$ of $\mc M_c$,
\begin{equation}
\label{n-50}
\begin{aligned}
& \limsup_{T\to\infty} \frac{\log T} T \log \bb P_\alpha
[ \bar \mu^T \in A] \\
&\qquad \le \; - \, \inf_{\mu \in A} 
\sup_{\gamma, \delta \le \varrho}
\min \Big\{ J^{\mc Q}_{\gamma, \delta} (\mu) 
\,-\, C (\gamma)\, \varrho \, , \,  
- C_\gamma (\delta , \varrho) \Big\} \;,
\end{aligned}
\end{equation}
where $J^{\mc Q}_{\gamma, \delta}$ is the functional given by
\eqref{n-49}. 

It remains to let $\delta\to 0$. Since $J_\gamma$ is lower
semi-continuous and since $\mu_\delta \to \mu$ as $\delta \to 0$, for
all $\mu\in \mc M_c$,
\begin{equation*}
J^{\mc Q}_{\gamma} (\mu) \;\le\;
\limsup_{\delta\to 0} J^{\mc Q}_{\gamma, \delta} (\mu)\;,
\end{equation*}
where $J^{\mc Q}_{\gamma}$ is defined in \eqref{n-53}.  Hence, letting
$\delta\to 0$ in \eqref{n-50} and then $\varrho \to 0$, we conclude
that for all closed subsets $A$ of $\mc M_c$,
\begin{equation*}
\limsup_{T\to\infty} \frac{\log T} T \log \bb P_\alpha
[ \bar \mu^T \in A] \; \le \; - \, \inf_{\mu \in A} 
\sup_{\gamma} J^{\mc Q}_{\gamma} (\mu)
\;= \; - \, \inf_{\mu \in A} J^{\mc Q} (\mu) \;,
\end{equation*}
where $J^{\mc Q}$ is the functional given by \eqref{n-54}.  This is
the upper bound of the large deviations principle, in view of Lemma
\ref{n-l17} .

\section{ The lower bound}
\label{n-sec6}

We prove in this section the lower bound of the large deviations
principle. Most of the results are taken from \cite{cll} and are
repeated here in sake of completeness.

Consider a functional $\mc J : \mc M_c \to \bb R_+ \cup \{+\infty\}$.
A subset $\mc M^*$ of $\mc M_c$ is said to be $\mc J$-dense if for
each $\mu\in\mc M_c$ such that $\mc J(\mu) <\infty$, there exists a
sequence $\{\mu_n \in \mc M^* : n\geq 1\}$ converging vaguely to $\mu$
and such that $\lim_{n\to\infty} \mc J (\mu_n) = \mc J (\mu)$.

Denote by $\mc M^*_0$ the subset of $\mc M_c$ formed by the measures
in $\mc M_{0}$ whose density $m$ is smooth, bounded away from $0$ and
$1$, and for which $m'$ has a compact support in $(0,\infty)$. The
next result follows from the proof of \cite[Lemma 4.1]{blm}. 

\begin{lemma}
\label{n-l20}
Recall from \eqref{n-39} the definition of the functional $\mc Q$.
The set $\mc M^*_0$ is $\mc Q$-dense.
\end{lemma}

Let $\mc M^*_{0,\alpha} = \mc M^*_0 \cap \mc M_{0,\alpha}$. Fix a
measure $\mu$ in $\mc M^*_{0,\alpha}$, and denote its density by
$m$. Since $m'$ has support contained in $(0,1/2)$ and $m(r) = \alpha$
for $r\ge 1/2$, $m$ belongs to $C^2(\bb R_+, \alpha)$. Hence, $\mc
M^*_{0,\alpha}$ corresponds to the measures whose density belongs to
$C^2(\bb R_+, \alpha)$.

\begin{corollary}
\label{n-l21}
The set $\mc M^*_{0,\alpha}$ is $I_{Q,\alpha}$-dense.
\end{corollary}

\begin{proof}
Fix $\mu$ such that $I_{Q,\alpha}(\mu)<\infty$. By definition of
$I_{Q,\alpha}$, $\mu$ belongs to $\mc M_{0,\alpha}$ and
$I_{Q,\alpha}(\mu) = \pi \mc Q(\mu)$. By the previous lemma, there
exists a sequence $\nu_n \in \mc M^*_0$ such that $\nu_n \to \mu$ and
$Q(\nu_n) \to Q(\mu)$.  To prove the corollary it is therefore enough
to show that for every $\mu \in \mc M^*_0$ there exists a sequence
$\mu_n \in \mc M^*_{0,\alpha}$ such that $\mu_n \to \mu$ and $Q(\mu_n)
\to Q(\mu)$.

Fix such a measure $\mu(dr) = m(r)\, dr$. Since $\mu$ belongs to $\mc
M^*_0$,
\begin{equation*}
Q(\mu) \;=\; \frac 14 \int_0^\infty \frac{[m'(r)]^2}{\sigma(m(r))}
\, dr\;.
\end{equation*}
Fix $\delta>0$, and let $u_\delta (r) = m(r) \, \bs 1\{ r \le (1/2) -
\delta \} + \alpha \, \bs 1\{ r > (1/2) - \delta \}$. Extend
$u_\delta$ to $(-\infty,0)$ by setting $u_\delta(r)= m(0)$ for $r\le
0$.  Let $m_\delta = u * \varphi_{\delta/2}$, where
$\varphi_{\delta/2}$ is a smooth approximation of the identity whose
support is contained in $[-\delta/2,\delta/2]$. Denote by $\mu_\delta$
the measure on $\bb R_+$ whose density is $m_\delta$.

It is clear that $\mu_\delta$ belongs to $\mc M^*_{0,\alpha}$ for
$\delta$ sufficiently small and that $\mu_\delta \to \mu$ as $\delta
\to 0$. By the lower semicontinuity of $\mc Q$, $\mc Q(\mu) \le
\liminf_{\delta\to 0} \mc Q(\mu_\delta)$. On the other hand, by
construction, for $\delta$ sufficiently small,
\begin{equation*}
\mc Q(\mu_\delta) \;=\; \int_0^{1/2} 
\frac{[m_\delta'(r)]^2}{\sigma(m_\delta(r))}\, dr \;\to\;
\int_0^{1/2} 
\frac{[m'(r)]^2}{\sigma(m(r))}\, dr \;\le\; \mc Q(\mu)\;. 
\end{equation*}
Hence, $\limsup_{\delta\to 0} \mc Q(\mu_\delta) \le \mc Q(\mu)$, which proves
the corollary.
\end{proof}

We are now in a position to prove the lower bound. We start with a law
of large numbers for the polar empirical measure under the measure
$\bb P_{T, \gamma}$. This result is Lemma 6.1 in \cite{cll}. It
follows from the stationarity of the measure $\nu_{T, \gamma}$ and
from the fact that it is a product measure.

\begin{lemma}
\label{m1}
Fix $\gamma$ in $C^2(\bb R_+,\alpha)$.  As $T\uparrow\infty$, the
measure $\bar \mu^T$ converges in $\bb P_{T, \gamma}$-probability to
the measure $\gamma (r) \, dr$.
\end{lemma}

\smallskip\noindent{\bf Proof of the lower bound.} We reproduce the
proof presented in \cite{cll}.  Fix an open subset $G$ of $\mc M_c$.
In view of Corollary \ref{n-l21}, it is enough to show that
\begin{equation*}
\liminf_{T\to\infty} \frac{\log T} T \log \bb P_{\alpha}
\big[ \bar\mu^T \in G \big] \;\ge\; - I_{Q,\alpha} (\mu)
\end{equation*}
for every $\mu$ in $\mc M^*_{0,\alpha} \cap G$. Fix such a measure
$\mu$ and denote its density by $\gamma$. As observed above the
statement of Lemma \ref{n-l21}, $\gamma$ belongs to $C^2(\bb R_+,
\alpha)$. Let $\mc A_\gamma =\{ \bar\mu^T \in G\}$ and denote by $\bb
P_{T, \gamma}^A$ the probability measure $\bb P_{T, \gamma}$
conditioned on the set $\mc A_\gamma$.  With this notation we may
write
\begin{equation*}
\bb P_{\alpha} [ \bar\mu^T \in G ] \;= \;
\bb E_{T, \gamma} \Big[ \frac{d\bb P_{\alpha}}{d\bb P_{T, \gamma}}
\mb 1\{ \mc A_\gamma \} \Big] \;=\;
\bb E_{T, \gamma}^A \Big[ \frac{d\bb P_{\alpha}}{d\bb P_{T, \gamma}} \Big]\,
\bb P_{T, \gamma} [ \mc A_\gamma ] \;.
\end{equation*}
By the law of large numbers stated in Lemma \ref{m1},
$\lim_{T\to\infty} \bb P_{T, \gamma} [ \mc A_\gamma ] =1$.
Hence, by Jensen inequality,
\begin{align*}
& \liminf_{T\to \infty} \frac {\log T} T \log \bb P_{\alpha} [
\bar\mu^T  \in G \Big] \; = \;
\liminf_{T\to \infty} \frac {\log T} T \log \bb E_{T, \gamma}^A
\Big[ \frac{d\bb P_{\alpha}}{d\bb P_{T, \gamma}} \Big] \\
&\quad \ge\; \liminf_{T\to \infty} \frac {\log T} T \bb E_{T, \gamma}^A
\Big[ \log \frac{d\bb P_{\alpha}}{d\bb P_{T, \gamma}} \Big]
\; =\; \liminf_{T\to \infty} \frac {\log T} T \bb E_{T, \gamma}
\Big[ \log \frac{d\bb P_{\alpha}}{d\bb P_{T, \gamma}}
\mb 1\{ \mc A_\gamma \} \Big]\; .
\end{align*}
By the bound \eqref{n-57} for the Radon-Nikodym derivative and by
Lemma \ref{m1}, last term is equal to
\begin{equation*}
\liminf_{T\to \infty} \frac {\log T} T \bb E_{T, \gamma}
\Big[ \log \frac{d\bb P_{\alpha}}{d\bb P_{T, \gamma}} \Big]  
\end{equation*}
which is, up to a sign, the entropy of $\bb P_{T, \gamma}$ with
respect to $\bb P_{\alpha}$. In view of formula \eqref{n-58} for the
Radon-Nikodym derivative $d\bb P_{T, \gamma} / d\bb P_{\alpha}$, the
previous limit is equal to
\begin{align*}
\liminf_{T\to \infty} \bb E_{T, \gamma}
\Big[ \pi \int_{\bb R_+} \Gamma''(r) \, \bar \mu^T(dr) \Big]
\; + \; \liminf_{T\to \infty} \bb E_{T, \gamma} \Big[
\int_0^1 W_\gamma (\eta_s) \, ds  \Big]\; ,
\end{align*}
where $W_\gamma$ is defined in \eqref{n-59}.  Since $\nu_{T, \gamma}$
is a stationary state, these expectations are easily computed. Recall
Lemma \ref{n-l22} to show that the limit is equal to
\begin{equation*}
\pi \int_0^{1/2} \Big\{ \Gamma''(r) \, \gamma(r)
+ [\Gamma'(r)]^2 \, \sigma(\gamma(r))\Big\} \, dr\,  \;=\; -\, \frac{\pi }4 \int_0^{1/2} 
\frac{[\gamma']^2}{\gamma (1-\gamma)}\, dr \; =\; -\, I_{Q,\alpha} (\mu)\;.  
\end{equation*}
We were allowed to integrate by parts the first term on the right-hand
side because the function $\Gamma = (1/2) \{\log \gamma/(1-\gamma) -
\log \alpha/(1-\alpha)\}$ vanishes at the boundary. This proves the
lower bound.

\end{document}